\documentclass[12pt]{article}

\usepackage[margin=1.25in]{geometry}

\usepackage{amssymb,amsmath}
\usepackage{amsthm}
\usepackage{hyperref}
\usepackage{mathrsfs}
\usepackage{textcomp}
\hypersetup{
    colorlinks,%
    citecolor=black,%
    filecolor=black,%
    linkcolor=black,%
    urlcolor=black
}

\usepackage{verbatim}

\usepackage{sectsty}

\usepackage[normalem]{ulem} 

\makeatletter

\newdimen\bibspace
\renewenvironment{thebibliography}[1]{%
 \section*{\refname 
       \@mkboth{\MakeUppercase\refname}{\MakeUppercase\refname}}%
     \list{\@biblabel{\@arabic\c@enumiv}}%
          {\settowidth\labelwidth{\@biblabel{#1}}%
           \leftmargin\labelwidth
           \advance\leftmargin\labelsep
           \itemsep\bibspace
           \parsep\z@skip     %
           \@openbib@code
           \usecounter{enumiv}%
           \let\p@enumiv\@empty
           \renewcommand\theenumiv{\@arabic\c@enumiv}}%
     \sloppy\clubpenalty4000\widowpenalty4000%
     \sfcode`\.\@m}
    {\def\@noitemerr
      {\@latex@warning{Empty `thebibliography' environment}}%
     \endlist}

\makeatother

\makeatletter

\newtheorem{thm}{Theorem}[section]
\newtheorem{lem}[thm]{Lemma}
\newtheorem{prop}[thm]{Proposition}

\newtheorem{rem}[thm]{Remark}

\def\Xint#1{\mathchoice
  {\XXint\displaystyle\textstyle{#1}}%
  {\XXint\textstyle\scriptstyle{#1}}%
  {\XXint\scriptstyle\scriptscriptstyle{#1}}%
  {\XXint\scriptscriptstyle\scriptscriptstyle{#1}}%
  \!\int}
\def\XXint#1#2#3{{\setbox0=\hbox{$#1{#2#3}{\int}$}
  \vcenter{\hbox{$#2#3$}}\kern-.5\wd0}}

\def\dashint{\Xint-}

\newcommand{\al}{\alpha}                \newcommand{\lda}{\lambda}
\newcommand{\om}{\Omega}                \newcommand{\pa}{\partial}
\newcommand{\va}{\varepsilon}           \newcommand{\ud}{\mathrm{d}}
\newcommand{\be}{\begin{equation}}      \newcommand{\ee}{\end{equation}}
                 
\newcommand{\Lda}{\Lambda}              
\newcommand{\R}{\mathbb{R}}              \newcommand{\Sn}{\mathbb{S}^n}

\providecommand{\keywords}[1]{\textbf{Keywords} \textit{ #1}}

\begin{document}

\title{\textbf{Asymptotic behavior of solutions to the Yamabe equation with an asymptotically flat metric}
\footnote{Accepted for Publication on 4/20/2023  in Journal of Functional Analysis, https://doi.org/10.1016/j.jfa.2023.109982. ©2023. This accepted manuscript version is made available under the \href{https://creativecommons.org/licenses/by-nc-nd/4.0/}{CC-BY-NC-ND 4.0 license.}}}
\bigskip

\author{\medskip
Zheng-Chao Han,\ \ \ Jingang Xiong\footnote{J. Xiong was partially supported by the National Key R\&D Program of China No.\ 2020YFA0712900 and NSFC 11922104 and 12271028.},\  \ \
Lei Zhang \footnote{L. Zhang was partially supported by a collaboration grant of Simons Foundation (Award Number: 584918).}}

\date{ }

\maketitle

\begin{abstract} We prove that any positive solution of the Yamabe equation on an asymptotically flat $n$-dimensional manifold of flatness order at least $\frac{n-2}{2}$
and $n\le 24$ must converge at infinity either to a fundamental solution of the Laplace operator on the Euclidean space or to a radial Fowler solution defined on the entire Euclidean space. The flatness order $\frac{n-2}{2}$ is the minimal flatness order required to define ADM mass in general relativity; the dimension $24$ is the dividing dimension of the validity of compactness of solutions to the Yamabe problem. We also prove such alternatives  for bounded solutions when $n>24$.

 We prove these results by establishing appropriate asymptotic behavior
near an isolated singularity of solutions to the Yamabe equation when
the metric has a flatness order
of at least  $\frac{n-2}{2}$ at the singularity and $n\le 24$, also when $n>24$ and
the solution grows no faster than the fundamental solution of the
flat metric Laplacian at the singularity. These results extend earlier results of
L. Caffarelli, B. Gidas and J. Spruck, also of N. Korevaar, R. Mazzeo, F. Pacard and R. Schoen,
 when the metric is conformally flat,
 and work of C.C. Chen and C. S. Lin when the scalar curvature is a non-constant function with appropriate flatness at the singular point,  also work of F. Marques
when the metric is not necessarily conformally flat but smooth, and  the dimension of the manifold is
three, four, or five, as well as recent similar results by the second and third authors  in dimension six.
\end{abstract}


\section{Introduction}

On a compact smooth Riemannian manifold $(M, g)$ of dimension $n\ge  3$, the Yamabe problem, which concerns the existence of constant scalar curvature metrics in the conformal class of $g$, was solved affirmatively through Yamabe \cite{Y}, Trudinger \cite{T}, Aubin \cite{A} and Schoen \cite{S}. The problem is equivalent to solving the Yamabe equation
\[
-L_g u= n(n-2) \mathrm{sign}(\lda_1) u^{\frac{n+2}{n-2}}  \quad \mbox{on }M, \quad u>0,
\]
where $L_g = \Delta_g- c(n)R_g $ is the conformal Laplacian with $c(n)= \frac { (n-2) }{ 4(n-1) }$, $\Delta_g$ is the Laplace-Beltrami operator and $R_g$ is the scalar curvature of $g$, and $\mathrm{sign}(\lda_1)\in \{-1,0,1\}$ is the sign of the first eigenvalue of $-L_g$ on $M$.

Solutions of the Yamabe equation on the standard unit sphere $\Sn$ were classified by Obata \cite{O}. Namely,  they must be positive constants ($2^{-\frac{n-2}{2}}$ in our formulation)
modulo Mobius transforms.
The same conclusion was proved on $\Sn\setminus \{\mathcal{N}\}$ by Gidas-Ni-Nirenberg \cite{GNN79, GNN81} and   Caffarelli-Gidas-Spruck \cite{CGS}, where $\mathcal{N}$ is the north pole. Equivalently, the theorem on $\Sn\setminus \{\mathcal{N}\}$ asserts that every positive solution of the Yamabe equation with the flat background metric
\be \label{eq:flat-CGS}
-\Delta u= n(n-2)u^{\frac{n+2}{n-2}}
\ee
in $\R^n$ has to be the form $\lda^{\frac{n-2}{2}}(1+\lda^2|x -x_0|^2)^{\frac{2-n}{2}}$, where $\lda>0$ and $x_0\in \R^n$. This Liouville type theorem implies that there is no positive solution of the Yamabe equation on $\Sn$ which is singular only at one point.

In the same paper  \cite{CGS}, Caffarelli, Gidas and Spruck further studied  the isolated singularities of positive solutions of \eqref{eq:flat-CGS}. First, they classified all positive solutions of \eqref{eq:flat-CGS} in $\R^n\setminus \{0\}$ (or $\Sn\setminus\{\mathcal{N}, -\mathcal{N}\}$) with $0$ being a non-removable singularity by proving  that they are radially symmetric and solve an ODE studied by Fowler \cite{F}. We refer these radial singular solutions
on $\R^n\setminus \{0\}$ as \textit{Fowler solutions}. The radial symmetry was also obtained by \cite{GNN81}  under some decay assumption on the solution at infinity.
Second, they proved that every positive solution of \eqref{eq:flat-CGS} in the punctured unit ball $B_1\setminus \{0\}$ with $0$ being a non-removable singularity must converge to a Fowler solution:
\be \label{eq:CGS}
u(x)=u_0(x)(1+o(1))\quad \mbox{as }x \to 0,
\ee
where $u_0$ is a Fowler solution which can be written as
\[
u_0(|x|)= |x|^{-\frac{n-2}{2}} \psi(\ln |x|)
\] and $\psi$ is a positive periodic solution on $\R$ to
$\psi''(t)-\frac{(n-2)^{2}}{4}\psi(t)+n(n-2)\phi(t)^{\frac{n+2}{n-2}}=0$.  A different proof and refinement of results in  \cite{CGS} were given by  Korevaar-Mazzeo-Pacard-Schoen \cite{KMPS};
 in particular, they improved the $o(1)$ remainder term to  $O(|x|^{\al})$ for some $\al>0$.
See also  Chen-Lin \cite{ChenLin1} and Li \cite{Lc} for alternative arguments and generalizations of  \cite{CGS} in establishing the upper bound of $u$, and Q. Han-Li-Li \cite{HLL} for higher order expansion of  $O(|x|^{\alpha})$.

It is a natural problem to ask whether the theorem of Caffarelli, Gidas and Spruck \cite{CGS} on a  punctured ball still holds when the background metric is not conformally flat. This problem was solved positively, when the background metric is smooth, by Marques \cite{f-mar} for $3\le n\le 5$ and by Xiong-Zhang \cite{x-z-1} for $n=6$ recently.  In this paper, we prove that the theorem
of \cite{CGS}, as described by \eqref{eq:CGS}, continues to hold for
$3\le n\le 6$ as well as for $7\le n \le 24$, even if the metric $g$ is not smooth across $\{0\}$,
as long as it has an expansion at $0$ of flatness of order $\tau\ge \frac{n-2}{2}$, as given by
\eqref{eq:inner-flat} below.

The problem is also formulated as the study of  asymptotic behavior at $\infty$
of  solutions of the Yamabe equation with an asymptotically flat background metric.
\begin{equation}\label{main-eq}
-L_gu=n(n-2)u^{\frac{n+2}{n-2}}\quad \mbox{in }\mathbb R^n\setminus B_1, \quad u>0,
\end{equation}
where $g$ is a smooth Riemannian metric defined on $\R^n\setminus \bar B_1$  satisfying the  standard asymptotically flat condition
\be \label{eq:asym-flt}
\sum_{i,j=1}^n\left |\nabla ^m\Big(g_{ij}(x)-\delta_{ij}\Big)\right |\le C|x|^{-\tau-m} \quad \mbox{for }x\in \mathbb R^n\setminus B_1,
\ee
where $\tau\ge \frac{n-2}{2}$,  $C$ is a positive constant and $m=0,\dots, n+2$.  Note that $\tau> \frac{n-2}{2}$ is necessary to define the ADM mass in general relativity;  see Denisov-Solov'ev \cite{DS} and Bartnik \cite{bartnik}.  This exterior formulation \eqref{main-eq} is equivalent to the punctured unit ball setting \cite{CGS} if $g$ is flat, since \eqref{eq:flat-CGS} is invariant under Kelvin transforms.

Our main results are as follows.

\begin{thm}\label{thm:main}
Let $u$ be a solution of (\ref{main-eq}) with the metric $g$ satisfying the asymptotic flatness \eqref{eq:asym-flt}. If $3\le n\le 24$, then either there exists a positive constant $a$
such that
$$u(x)=a|x|^{2-n}+O(|x|^{1-n})
 \quad \mbox{as }x \to \infty $$
or there exist $\alpha \in (0,1)$ and a Fowler solution $u_0$ such that
$$u(x)=u_0(|x|)(1+O(|x|^{-\alpha})\quad \mbox{as }x \to \infty.  $$
\end{thm}

Both of these alternatives can happen. In higher dimensions, we have
\begin{thm}\label{thm:main-25}
Let $u$ be a solution of (\ref{main-eq}) with the metric $g$ satisfying the asymptotic flatness (\ref{eq:asym-flt}).
If $n\ge 25$ and
\be\label{eq:25-wk-bound}
\limsup_{x\to \infty} u(x)<\infty,
\ee
then the conclusion of Theorem \ref{thm:main} still holds.
\end{thm}

The existence of complete conformal metrics of positive constant scalar curvature on
 $\Sn\setminus \Lda$ and related problems have been studied by Schoen \cite{S1}, Schoen-Yau \cite{SY}, Mazzeo-Smale \cite{MS},  Mazzeo-Pollack-Uhlenbeck \cite{MPU}, Mazzeo-Pacard \cite{MP}, Bettiol-Piccione-Santoro \cite{BPS} and others, where $\Lda$ is a closed set. It was proved in \cite{SY} that the Hausdorff dimension of $\Lda$ is necessary to be at most $\frac{n-2}{2}$. See Byde \cite{Byde} and Silva Santos \cite{SS} on other manifolds.

 When the scalar curvature is negative, the first study of this problem goes back to Loewner-Nirenberg \cite{LoN}, where the dimension of $\Lda$ was proved to be at least $\frac{n-2}{2}$. Later work on the negative scalar curvature case was done on general manifolds by Aviles-McOwen \cite{AMc88}, Finn-McOwen \cite{FMc}, Andersson-Chr\'usciel-Friedrich \cite{ACF}, Mazzeo \cite{Mz} and etc. See also the recent paper Q. Han-Shen \cite{HS} and references therein for related work.

By performing a Kelvin transform,  \eqref{main-eq} is equivalent to
\begin{equation}\label{eq:main-isolated}
-L_{\hat g}u= n(n-2)u^{\frac{n+2}{n-2}}\quad \mbox{in } B_1\setminus\{0\} , \quad u>0,
\end{equation}
and \eqref{eq:asym-flt} translates into
\be
\label{eq:inner-flat}
|\nabla^m(\hat g_{ij}(x)-\delta_{ij})|\le C |x|^{\tau-m} \quad\text{in $B_1\setminus\{0\}$ for $0\le m\le n+2$.}
\ee
 Note that the scalar curvature $R_{\hat g}$ might not be continuous across $0$ if $n\le 6$, but
 we still have $|R_{\hat g}(x)|\le  C |x|^{\tau-2}$ in $B_1\setminus\{0\}$,  which implies that
 $|R_{\hat g}(x)|\in L^p(B_1)$ for some $p>n/2$. The general case of the De Giorgi-Nash-Moser theorem
 implies that any solution $u$ of \eqref{eq:main-isolated} which is
 {\it bounded} in $B_1\setminus\{0\}$ must extend to a H\"older continuous solution in $B_1$,
 and as a supersolution
 of $L_{\hat g} u= 0$ in $B_1$, $u(0)=\lim_{x\to 0}u(x)>0$, see, e.g., \cite{Stampacchia}.
 We refer to such a situation as $x=0$ being a  removable singularity, and the solution $u(x)$ as a classical solution.
The assumption \eqref{eq:25-wk-bound} in Theorem \ref{thm:main-25} for $n\ge 25$  translates  into
\be \label{eq:25-wk-bd-1}
\limsup_{x\to 0} |x|^{n-2} u(x)<\infty.
\ee

\begin{thm} \label{thm:pre-final}
Suppose that $u$ is a solution of \eqref{eq:main-isolated} with $\hat g$ satisfying \eqref{eq:inner-flat}.
 If $n\ge 25$, we assume further that $u$ satisfies \eqref{eq:25-wk-bd-1}.
 Then either $0$ is a removable singularity, or there exists a constant $C>0$ such that
\be \label{eq:up-low}
\frac{1}{C} |x|^{-\frac{n-2}{2}}\le u(x)\le C|x|^{-\frac{n-2}{2}} \quad \mbox{for }x\in B_{1/2}\setminus \{0\}.
\ee

Moreover, we have $P(u)\le 0$, and $P(u)=0$ if and only if $0$ is a removable singularity, where $P(u)$ is the Pohozaev limit of $u$ defined in \eqref{eq:pohozaev-lim}.
\end{thm}

\begin{thm} \label{thm:pre-final-1} Under the same assumptions as in Theorem \ref{thm:pre-final}, if $0$ is not a removable singularity, then there exist $\alpha \in (0,1)$ and a Fowler solution $u_0$ such that
$$u(x)=u_0(|x|)(1+O(|x|^{-\alpha})\quad \mbox{as }x \to 0.  $$

\end{thm}

Based on previous work in this area, a key in proving Theorems \ref{thm:main}, \ref{thm:main-25}, \ref{thm:pre-final}, \ref{thm:pre-final-1} is to prove, for an unbounded solution $u$ of \eqref{eq:main-isolated}, \eqref{eq:up-low} holds.

When $\hat g$ is conformally flat, the proofs of the upper bound in \eqref{eq:up-low} of Caffarelli, Gidas and Spruck \cite{CGS} and Korevaar-Mazzeo-Pacard-Schoen \cite{KMPS} rely on inversion symmetries of \eqref{eq:flat-CGS},  using different variants of the Alexandrov reflection. When the background metric is not conformally flat,
that approach is no longer directly applicable, as
$u-u_{reflection}$ now satisfies an inhomogeneous linear elliptic equation
 and thus a key ingredient of the reflection argument (maximum principle) is not directly available.

 Chen-Lin \cite{ChenLin2} introduced an idea of constructing suitable auxiliary functions,
 when employing the moving planes method, to compensate for the loss of invariance
 under Mobius transformations of  the prescribing scalar curvature equation
\[
-\Delta u= K u^{\frac{n+2}{n-2}} \quad \mbox{in }B_1
\]
when $K$ is a positive non-constant continuous function---they used the method
to deal with solutions with or without isolated singularities.  For the case of
a solution with an isolated singularity at $0$,
they obtained a sharp flatness criterion of  $K(x)$  in another paper \cite{ChenLin3}
 to have the asymptotic behavior \eqref{eq:CGS} for isolated singularities---their
 criterion turns out to be  $\frac{1}{C} |x|^{l-1} \le   |\nabla K(x)|\le   C|x|^{l-1}$
for some constants $C\ge 1$ and  $l\ge \frac{n-2}{2}$.
 A counterexample was constructed when the flatness is less than $\frac{n-2}{2}$.
 See also Zhang \cite{Z}, Taliaferro-Zhang \cite{talia} and Lin-Prajapat \cite{Lin-P-1, Lin-P-2}.
 This idea of constructing auxiliary functions was adapted and developed to study the compactness of solutions to the Yamabe equation by Li-Zhang \cite{zhang-li-low-dim, LZ2, LZ3} via the moving spheres method, and to study isolated singularities of the Yamabe equation by previously mentioned paper Marques \cite{f-mar} and Xiong-Zhang \cite{x-z-1} for $n\le 6$.
The situation in these papers
is somewhat different from that of the prescribing scalar curvature equation in that
one can not impose some condition analogous to  the positive lower bound
 $ |\nabla K(x)|\ge \frac{1}{C} |x|^{l-1}$ as in \cite{ChenLin3}.

The auxiliary functions in \cite{f-mar} and \cite{x-z-1} are radially symmetric.
In our situation, we need to
construct non-radial auxiliary functions and prove some needed quantitative estimates.
A major part of these estimates is
 proved by applying and refining the blow up analysis developed in the studies of
 compactness of solutions to the Yamabe problem by  \cite{S3, LZhu99, D, f-mar-1, LZ2, LZ3, KMS}
 up to dimension $24$. In particular, for $n\ge 7$,
 we adapt some arguments from Li-Zhang \cite{LZ2, LZ3} and
 Khuri-Marques-Schoen \cite{KMS}, where
 the spectral analysis of the linearized Yamabe equation at the spherical solutions played
 an important role. However, our situation is  different from theirs, as the solutions
  and the metric  contain singular points and the blow up analysis is implemented near those points.
    If $n\ge 25$, the assumption \eqref{eq:25-wk-bd-1} is used to ensure the desired sign of
 the Pohozaev integral in the notation of \cite{KMS} for some specific blowing up sequence of solutions.
  Note that the compactness of solutions to the Yamabe problem on smooth compact manifolds
  which are not conformal diffeomorphic to the unit sphere fails in dimension $n\ge 25$,
  see Brendle \cite{Brendle} and Brendle-Marques \cite{BM}.

 Once we have the upper bound, the lower bound in \eqref{eq:up-low} is proved via the method
 of \cite{ChenLin2} and \cite{f-mar}, based on an analysis of behavior of solutions of
 an ordinary differential inequality satisfied by the spherical average of the solution and
 the Pohozaev integral.  However, there is a non-trivial linear term in the
 differential inequality of the spherical average of solutions
 in our case, which causes a technical issue when $\tau=\frac{n-2}{2}$.
 When $n=6$ and the background metric is smooth,
 this issue was solved by finding a good conformal metric in \cite{x-z-1}. Here we prove
 a refined ODE type estimate to prove the lower bound in \eqref{eq:up-low}.

The results of  \cite{CGS} and \cite{KMPS} have been extended to some fully nonlinear Yamabe equations or higher order conformally invariant equations, see, for example, Li-Li \cite{LiLi},   Li \cite{Li}, Chang-Han-Yang \cite{CHY}, Han-Li-Teixeira \cite{HLT}, Jin-Xiong \cite{JX}, and the references therein.

\medskip

The paper is organized as follows.  In Section \ref{sec:2}, we reduce the exterior problem to \eqref{eq:main-isolated} and outline a proof of the upper bound of solutions with an isolated singularity in Theorem \ref{main-ub},
deferring the details of the construction of auxiliary functions used in the proof
to Section \ref{sec:4}. In Section \ref{sec:blow-up}, we recall some local blow up analysis
for smooth solutions of the Yamabe equation up to dimension $24$, and set up conformal normal like coordinates to be used later.  In Section \ref{sec:4}, we provide details in proving the upper bound in Theorem \ref{main-ub}.
We divide the construction of the auxiliary function used in the moving spheres argument
 into the cases $3\le n\le 6$ and $ n\ge 7$. The difficulty of  former case lies in the singularity of the metric while that of  the latter lies in high dimensional effect.
In Section \ref{sec:lb}, we prove the lower bound and give a criterion of removability in terms of the sign of the Pohozaev integral. In Section \ref{sec:ode}, we provide details of some improved ODE type estimates for the spherical average of the solution, which are used in Section \ref{sec:lb}.

\medskip

\noindent {\bf Acknowledgment:} This work was started when J. Xiong was visiting Rutgers University during the academic year 2019--2020, to which he is grateful for providing the very stimulating research environments and supports. All authors would like to thank
Professor YanYan Li for valuable discussions and comments. This work was posted on arXiv as arXiv:2106.13380v2.

\section{Reduction to an isolated singularity and outline proof of the upper bound of solutions with an isolated singularity}
\label{sec:2}

Let $u$ be a positive solution of \eqref{main-eq} with the metric  $g$ satisfying \eqref{eq:asym-flt}. We shall use an inversion to transform the problem into
one with an isolated singularity on the punctured unit ball.
For any $x\in \R^n\setminus B_1$, let $x= \frac{z}{|z|^2}$ for $z\in B_1\setminus \{0\}$.
Then we see that
\begin{align*}
u(x)^{\frac{4}{n-2}} g_{ij}(x)\ud x^i \ud x^j&=\Big(\frac{1}{|z|^{n-2}}u(\frac{z}{|z|^2})\Big)^{\frac{4}{n-2}} |z|^4 g_{ij}(\frac{z}{|z|^2}) \ud (\frac{z^i}{|z|^2})\ud (\frac{z^j}{|z|^2})\\&=:
v(z)^{\frac{4}{n-2}} \hat g_{kl}(z) \ud z^k \ud z^l ,
\end{align*}
where $v(z)= \frac{1}{|z|^{n-2}}u(\frac{z}{|z|^2})$, and
\begin{align}
 \hat g(z) := \hat g_{kl} (z)\ud z^k \ud z^l
 =\sum_{i,j} g_{ij}(\frac{z}{|z|^2})  (\delta_{ik}  - 2z^i z^k )
  (\delta_{il}-  2z^i z^l )  \ud z^k \ud z^l.
\label{eq:h-def}
\end{align}
By \eqref{eq:asym-flt}, we have
\[             
\sum_{k,l=1}^n\left |\nabla ^m\Big(\hat g_{kl}(z)-\delta_{kl}\Big)\right |\le C_0 |z|^{\tau-m} \quad \mbox{for }z\in  B_1 \setminus \{0\},
~ m=0,1,...,n+2,
\]
which is the same as \eqref{eq:inner-flat}. By the conformal invariance of $L_g$,
\be \label{eq:main-u1}
-L_{\hat g}v= n(n-2) v^{\frac{n+2}{n-2}} \quad \mbox{in }B_1\setminus \{0\}.
\ee
Hence, the study of solutions of \eqref{main-eq} with the metric  $g$ satisfying \eqref{eq:asym-flt} has been
reduced to the study of solutions of \eqref{eq:main-u1} with  the metric  $\hat g$ satisfying \eqref{eq:inner-flat}.  If the assumption \eqref{eq:25-wk-bound} in Theorem \ref{thm:main-25} holds, then
it implies the upper bound \eqref{eq:25-wk-bd-1} for $v$:
\[
\limsup_{z\to 0}|z|^{n-2} v(z) <\infty.
\]

To avoid using too many variables, from now on we will rename the $z$ variable in $B_1\setminus \{0\}$
as $x$, $\hat g$ as $g$, and $v(z)$ as $u(x)$, namely, we will study $u(x)$, which satisfies \eqref{eq:main-isolated}.

\begin{thm}\label{main-ub} Let $u$ be a solution of \eqref{eq:main-isolated} with the metric $g$ satisfying \eqref{eq:inner-flat}.
  When $n\ge 25$, suppose that \eqref{eq:25-wk-bd-1} holds.  Then
\begin{equation}\label{sphe-har-e}
\limsup_{x\to 0}|x|^{\frac{n-2}2} u(x)<\infty.
\end{equation}
\end{thm}

When $g$ is a smooth metric in $B_1$, Theorem \ref{main-ub} was proved by Marques \cite{f-mar} for $n\le 5$ and by Xiong-Zhang \cite{x-z-1} for $n=6$.

Our proof of Theorem \ref{main-ub} follows the classical approach initiated by Schoen,
further developed by many authors over the last three  decades as described earlier;
we do need to overcome difficulties caused by
the potential singularity of the metric at $0$
when $3\le n\le 6$ and the high dimensional effect when $ n\ge 7$.
We will discuss the technical aspects of this analysis a bit later, much of which  will follow the approach in \cite{x-z-1}.

For now, we will first outline the setup for proving
Theorem \ref{main-ub}. Clearly, it suffices to consider
\[
\tau=\frac{n-2}2.
\]  If \eqref{sphe-har-e} were invalid,  there would exist a sequence $x_k\to 0$ such that
\be\label{eq:contra-hy-1}
|x_k|^{\frac{n-2}2} u(x_k)\to \infty \quad \mbox{as }k\to \infty.
\ee

\begin{lem}\label{lem:blow-up} The sequence  $x_k$ in \eqref{eq:contra-hy-1} can be selected to be local maximum points of $u$. Moreover,  there exists $0<\al_k<|x_k|/2$ such that
\[
u(x)\le 2^{\frac{n-2}{2}}  u(x_k) \quad \forall~x\in B_{\al_k}(x_k)
\]
and
\[
\lim_{k\to \infty} u(x_k) \al_k^{\frac{n-2}{2}} =\infty.
\]
\end{lem}
\begin{proof}
The proof is standard by now. See page 3 of \cite{x-z-1} for details.
\end{proof}

We will do a blow up analysis of $u(x)$ near $x_k$. To carry out a more refined analysis needed later on,
we will work with a conformal normal like coordinate system $\{ z \}$ centered at $x_k$. The precise formulation
will be given as  Lemma \ref{lem:conf-normal} in section \ref{sec:blow-up}.  The key is to set up this coordinate system
on a common ball centered at $z=0$ for a certain neighborhood in metric $g$ ``centered at" $x_k$
which includes $B_{\frac 12}(x_k)$, via a coordinate map $\phi_{x_k}(z)$ with appropriate control,
and with $\phi_{x_k}(0)=x_k$, $\phi_{x_k}(-x_k)=0$.

  We will work with $u_k(z)=(\kappa_{x_k} u) \circ \phi_{x_k}(z)$
for $z\in B_{1/2}\setminus \{-x_k\}$,
where $\kappa_{x_k}, \phi_{x_k}$, and $g_{k}(z)$ are defined as in Lemma \ref{lem:conf-normal}, with $x_{k}$ as $\bar z$ there, \eqref{eq:varkappa} holds for $g_{k}(z)$, and
\be  \label{eq:coord-u}
\det g_{k}(z) =1 \quad \mbox{in }B_{\sigma |x_k|}.
\ee
As usual, the analysis in this work uses multiple rescalings and
multiple coordinate systems. To avoid excessive notations, we may reuse the same variable
names in different contexts. For example $x_k$ is selected in the $x$ coordinate in $B_1\setminus\{0\}$
for $u$ according to Lemma \ref{lem:blow-up}, but in the conformal normal like coordinate $z$,
the $x=0$ point corresponds to $z=-x_k$, so we are using $x_k$ in two different coordinate system.

According to Lemma \ref{lem:conf-normal},
$u_k(z)$ is well defined for $z\in B_{1/2}\setminus \{-x_k\}$ for large $k$, and
\be \label{eq:u-conformal-normal}
-L_{g_k} u_k(z)= n(n-2)u_k(z)^{\frac{n+2}{n-2}} \quad \mbox{in }B_{1/2}\setminus \{-x_k\}.
\ee

Let
$$M_k= u_k(0), \quad l_k=M_k^{\frac{2}{n-2}}, \quad S_k=-l_kx_k, $$
and
$$v_k(y)=M_k^{-1} u_k(l_k^{-1}y), \quad y\in B_{\frac{l_k}{2}} \setminus \{S_k\}. $$
 By Lemma \ref{lem:blow-up},
$\lim_{k\to \infty}|S_k|= \infty$. By  equation \eqref{eq:u-conformal-normal} of $u_k$, we have
\begin{equation}\label{v-k}
-L_{\bar g_k}v_k(y)=n(n-2)v_k(y)^{\frac{n+2}{n-2}}\quad \mbox{in }B_{\frac{l_k}{2}} \setminus \{S_k\},
\end{equation}
where $(\bar g_k)_{ij}(y) =(g_k)_{ij}(l_k^{-1}y)$.

By Lemma \ref{lem:blow-up} and Lemma \ref{lem:conf-normal}, we know that $x_k=\phi_{x_k}(0)$
is a critical point of $u$ and
$ \kappa_{x_k}$. Thus $\nabla v_k(0)=0$. It follows from local estimates for solutions of linear elliptic equations that,
up to passing to a subsequence,
\[
v_k(y) \to U(y) \quad \mbox{in }C_{loc}^{2} (\R^n)  \quad \mbox{as }k\to \infty
\]  for some nonnegative function $U\in C^2(\R^n)$ satisfying
$$\Delta U(y)+n(n-2)U(y)^{\frac{n+2}{n-2}}=0\quad \mbox{in } \mathbb R^n, $$
\[
U(0)=1, \quad \nabla U(0)=0.
\]
By the classification theorem of  Caffarelli-Gidas-Spruck \cite{CGS}, we have
\be \label{eq:the-bubble}
U(y)=(1+|y|^2)^{\frac{2-n}2}.
\ee
Later on, we may write $U(|y|)=U(y)$ without causing confusion.

Recall that the Kelvin transform of a function $f$ defined on some measurable set $\om \subset \R^n$
with respect to the sphere $\pa B_\lda$, $\lda>0$, is defined as
\[
f^\lda(y):= (\frac{\lda}{|y|})^{n-2} f(y^\lda), \quad \text{for $y$ such that $y^\lda:= \frac{\lda^2 y}{|y|^2}\in \om$.}
\]

Applying the Kelvin transform to $U(y)$, we note that $U^{\lambda} (y)=U(y)$ for all $y$ when $\lambda=1$, and
\[
\left\{
\begin{aligned}
U^{\lambda} (y) &<U(y) \quad\text{for  $0<\lambda<1$ and $y$ with $|y|>\lambda$, and}\\
U^{\lambda} (y) &>U(y) \quad\text{for $\lambda>1$ and $y$ with $|y|>\lambda$.}\\
\end{aligned}
\right.
\]
 In fact we have a more precise estimate
\begin{align} \label{eq:model-fact}
U(y)-U^{\lambda}(y)
\begin{cases}
&>c_1(n)(1-\lda)(|y|-\lambda)|y|^{1-n},\quad \mbox{if } 0<\lambda<1\mbox{ and } |y|>\lambda, \\
& \equiv 0,\quad \mbox{if } \lambda=1,\\
&<-c_1(n)(\lda-1 )(|y|-\lambda )|y|^{1-n},\quad \mbox{if } \lambda>1\mbox{ and } |y|>\lambda,
\end{cases}
\end{align}
where $c_1(n)>0$ is a dimensional constant.

Using \eqref{eq:model-fact} and $v_k(y) \to U(y)=(1+|y|^2)^{\frac{2-n}2}$ in $C_{loc}^{2} (\R^n)$,
we see that for $\lambda>1$, $y$ with $|y|>\lambda$,
\be \label{eq:vkup}
v_k(y)- v^{\lambda}_k(y)<-c_1(n)(\lda-1 )(|y|-\lambda )|y|^{1-n} \text{  for sufficiently large $k$.}
\ee
For any fixed $\lda'>\lda>1$ and $R>\lda'$,
we can choose a common $k'$ such that the above holds
for any $k\ge k'$ and $\lda'\le |y| \le R$; but  for our purpose it suffices that for
each fixed $|y|>\lda>1$, there exists some $k'$ such that the above holds for $k\ge k'$.

On the other hand
 we shall apply the method of moving spheres in $\Sigma^k_{\lambda}\setminus \{S_k\}$
to $w_{\lambda}(y)=v_k(y)-v_k^{\lambda}(y)$ (for simplicity we omit the subscript $k$ in $w_{\lambda}$)
 to show that
  \be \label{eq:vkd}
  v_k(y)-v^{\lambda}_k(y)+f_{\lda, k}(y)\ge 0\quad
  \text{for $y\in \Sigma^k_{\lambda}\setminus \{S_k\}$,}
  \ee
   for $1-\delta_1 \le \lambda \le 1+\delta_1$, and
   for sufficiently large $k$, where
 $$\Sigma^k_{\lambda}:=B_{Q_1 l_k^{1/2}}\setminus \bar B_{\lambda}, $$
  $f_{\lda, k}(y)$ is an auxiliary function constructed to make
 $v_k(y)-v^{\lambda}_k(y)+f_{\lda, k}(y)$ a supersolution of a linear operator \eqref{fsuper}
 related to $L_{\bar g_k}$ and $f_{\lda, k}(y)\to 0$ uniformly over
 $y\in\Sigma^k_{\lambda}$ as $k\to \infty$,
and  for the $3\le n\le 6$ cases $\delta_1$ can be taken as $\frac 12$, and for the $n\ge 7$ cases
$\delta_1\in (0,1/2)$ is determined to make
 it possible to construct the auxiliary function $f_{\lda}^{(3)}(y)$ for $1-\delta_1\le \lda \le 1+\delta_1$ from \eqref{f3sol}.
 Furthermore, according to  Proposition \ref{prop:x-k-estimates} to be proved later, $Q_1>0$ can be taken such that
\[
S_k\in B_{Q_1 l_k^{1/2}}.
\]
The construction of $f_{\lda, k}(y)$ with the desired properties is carried out in
Lemmas \ref{lem:h-est-1}, \ref{lem:h-est-2}, and Proposition \ref{prop:final}
for $\frac 12 \le \lda \le 2$ in the $3\le n \le 6$ cases, and respectively, in
Lemma~\ref{lem:f4}, \eqref{eq:724-e},  and Proposition~\ref{cor:barrier-final} for
$1-\delta_1 \le \lambda \le 1+\delta_1$ in the
$n\ge 7$ cases; and the moving spheres argument for the $n\ge 7$
requires a minor modification to $\Sigma_{\lda}^k$, as done in Proposition \ref{cor:barrier-final}.

 The two opposing estimates \eqref{eq:vkup} and \eqref{eq:vkd} on $v_k(y)- v^{\lambda}_k(y)$
 are obtained  under the assumption \eqref{eq:contra-hy-1}.
 We now fix some $1<\lda\le 1+\delta_1$ and some $y$ with $|y|>\lda$:
 on the one hand, for sufficiently large $k$, we have \eqref{eq:vkup}; on the other hand,
 $f_{\lda, k}(y)\to 0$ implies that for   sufficiently large $k$,
 $c_1(n)(\lda-1 )(|y|-\lambda )|y|^{1-n}>f_{\lda, k}(y)$, which,
 together with $v_k(y)-v^{\lambda}_k(y)+f_{\lda, k}(y)\ge 0$,  leads to
 $v_k(y)-v^{\lambda}_k(y)+c_1(n)(\lda-1 )(|y|-\lambda )|y|^{1-n}>0$.
  This contradiction
 shows that the scenario of \eqref{eq:contra-hy-1} can't happen.

To apply the method of moving spheres, we first need to work out the equations satisfied by $w_{\lambda}$.
We first record the  expansions for the coefficients in $L_{g_k}$ based on Lemma \ref{lem:conf-normal},
\be \label{eq:g-kk}
\begin{split}
L_{g_k} &=\Delta_{g_k} -c(n)R_{g_k}(z)\\&
=\Delta + \frac{1}{\sqrt{\det g_k(z)}}\pa_j \left(g^{ij}_k(z) \sqrt{\det g_k(z)}\right) \pa_i  +(g^{ij}_k(z)-\delta_{ij})\pa_{ij} -c(n)R_{g_k}(z) \\&=:\Delta +b_i(z) \pa_i+d_{ij}(z) \pa_{ij} -c(z),
\end{split}
\ee
with
\[
b_i(z)=
\begin{cases} O(|x_k|^{\tau-2})|z|,&\quad |z|< \sigma |x_k|,\\
O(|z+x_k|^{\tau-1}),&\quad |z|\ge \sigma |x_k|,
\end{cases}
\]
\[
d_{ij}(z)=\begin{cases} O(|x_k |^{\tau-2})|z|^2,&\quad |z|< \sigma |x_k|,\\
O(|z+x_k|^{\tau}),&\quad |z|\ge \sigma |x_k |,
\end{cases}
\]
\[
c(z)=\begin{cases} O(|x_k |^{\tau-4})|z|^2,&\quad |z|< \sigma |x_k|,\\
O(|z+x_k |^{\tau-2}),&\quad |z|\ge \sigma |x_k |,
\end{cases}
\]
where $\sigma>0$ is independent of $k$, and we drop the subscript $k$ of $b_i, d_{ij}$ and $c$ for brevity.

After the scaling $y=l_k z$, by \eqref{eq:g-kk} and the definition of $\bar g_k$ after \eqref{v-k}, we have
\be\label{eq:bar-g}
L_{\bar g_k}= \Delta+\bar b_i(y)\partial_i+\bar d_{ij}(y)\partial_{ij}-\bar c(y)
\ee
where the differentiations are with respect to $y$ and
\begin{align}\label{bik}
\bar b_i(y)=&l_k^{-1}b_i(l_k^{-1}y)\nonumber \\[2mm]
=& \begin{cases}
O(l_k^{-2}|x_k|^{\tau-2}|y|),&\quad |y|<\sigma  |S_k|,\\[2mm]
O(l_k^{-\tau}|y-S_k|^{\tau-1}),&\quad |y|\ge \sigma  |S_k|,
\end{cases}
\end{align}
\begin{align}\label{dijk}
\bar d_{ij}(y)&=d_{ij}(l_k^{-1}y)\nonumber \\[2mm]& =
\begin{cases}
O(l_k^{-2}|x_k|^{\tau-2}|y|^2),\quad |y|<\sigma  |S_k|,\\[2mm]
O(l_k^{-\tau}|y-S_k|^{\tau}),\quad |y|\ge \sigma  |S_k|,
\end{cases}
\end{align}
and
\begin{align}\label{ck}
\bar c(y)&=c(n)l_k^{-2}R_{g_k}(l_k^{-1}y)\nonumber \\&= \begin{cases}
O(l_k^{-4}|x_k|^{\tau-4}|y|^2),\quad |y|<\sigma  |S_k|,\\[2mm]
O(l_k^{-\tau}|y-S_k|^{\tau-2}),\quad |y|\ge \sigma |S_k|.
\end{cases}
\end{align}

Unlike in the locally conformally flat case, here, $v_k^{\lambda}(y)$ no longer satisfies an equation
of the exact same form as $v_k(y)$ so we can't directly apply the moving spheres method
to $w_{\lambda}(y)$. But a direct computation,
using the equation for $v_k(y)$ at $y$ and $y^{\lda}$,
as well as $\Delta v^{\lda}_k(y)=\left(\frac{\lda}{|y|}\right)^{n+2} \Delta v_k(y^{\lda})$,  yields
\begin{equation}\label{eq-w}
\left[\Delta +\bar b_i(y)\partial_i +\bar d_{ij}(y)\partial_{ij}-\bar c(y)\right]w_{\lambda}(y)
+\xi_\lda(y)w_{\lambda}(y)=E_{\lambda}(y),\quad y\in \Sigma^k_{\lambda}\setminus \{S_k\},
\end{equation}
where
\be  \label{eq:xi-def}
\xi_\lda (y)=\left\{\begin{array}{ll}
\displaystyle{n(n-2)\frac{v_k(y)^{\frac{n+2}{n-2}}-v_k^{\lambda}(y)^{\frac{n+2}{n-2}}}{v_k(y)-v_k^{\lambda}(y)}}, \quad \mbox{if }\quad v_k(y)\neq v_k^{\lambda}(y),\\[2mm]
n(n+2)v_k^{\frac{4}{n-2}}(y),\quad \mbox{if}\quad v_k(y)=v_k^{\lambda}(y),
\end{array}
\right.
\ee
and
\begin{eqnarray} \label{Elda}
E_{\lambda}(y)&=& \left(\bar c(y)v_k^{\lambda}(y)- (\frac{\lambda}{|y|})^{n+2}
\bar c(y^{\lambda})v_k(y^{\lambda})\right)
-\left(\bar b_i(y)\partial_i v_k^{\lambda}(y)+\bar d_{ij}(y)\partial_{ij}v^{\lambda}_k(y)\right)
\nonumber\\
&& +(\frac{\lambda}{|y|})^{n+2}\left(\bar b_i(y^{\lambda})
\partial_i v_k(y^{\lambda})+
\bar d_{ij}(y^{\lambda})\partial_{ij}v_k(y^{\lambda})\right).
\label{eQ}
\end{eqnarray}

What is going to make the moving spheres method work to prove Theorem \ref{main-ub}
is that we are able to construct some $f_{\lda}(y)$ (we have dropped the index $k$ for $f_{\lda}(y)$) such that
\[
\left[\Delta +\bar b_i(y)\partial_i +\bar d_{ij}(y)\partial_{ij}-\bar c(y)\right]f_{\lda}(y)
+\xi_\lda(y) f_{\lda}(y) \le -|E_{\lambda}(y)| \quad \mbox{in }\Sigma^k_{\lda}
\]
with control, including   $f_{\lda}(y)=0$ for $y\in \partial B_{\lda}$ and
\be \label{fest}
|f_\lda(y)| +|\nabla f_\lda(y)|=o(1)|y|^{2-n}
\text{ uniformly for } y\in   \Sigma^k_\lda.
\ee
We remark that \eqref{fest} for the
 $3\le n \le 6$ cases follows from Lemma \ref{lem:h-est-1} and Lemma \ref{lem:h-est-2},
  and for the $n\ge 7$ cases follows from Proposition \ref{cor:barrier-final}.
In fact, we need to modify $\Sigma^k_\lda$ slightly into $\tilde \Sigma^k_\lda$ for the $n\ge 7$ cases---see
Proposition \ref{cor:barrier-final} for the construction of $f_\lda(y)$ for the $n\ge 7$ cases and the definition
of $\tilde \Sigma^k_\lda$, and Proposition \ref{prop:final} for the $3\le n \le 6$ cases.
Using $f_{\lda}(y)$, we have
\be \label{fsuper}
\left[ \Delta +\bar b_i(y)\partial_i +\bar d_{ij}(y)\partial_{ij}-\bar c(y)\right][w_{\lambda}(y)+ f_{\lda}(y)]+\xi_\lda [w_{\lambda}(y)+f_{\lda}(y)]\le 0 \quad \mbox{in }\Sigma^k_{\lda}\setminus \{S_k\}.
\ee

 Our construction allows us to show that the moving spheres process can be started:
\begin{equation}\label{v-k-s}
w_{\lambda}(y)+f_{\lambda}(y)>0 \quad \mbox{in}\quad  \Sigma^k_{\lambda}\setminus \{S_k\} \quad \mbox{for } \lambda\in [1-\delta_1, 1-\delta_1/2],
\end{equation}

Indeed, for any $\lambda \in [1-\delta_1,1-\delta_1/2]$ (the arithmetic below is worked out
assuming $1-\delta_1/2\le  3/5$ to get a clean constant in the estimate), by \eqref{eq:model-fact} we have
$$U(y)-U^{\lambda }(y)>\frac{2c_1(n)}{5}(|y|-\lambda )|y|^{1-n}\quad \mbox{for } |y|>\lambda; $$
in addition, $\frac{y}{|y|}\cdot \nabla \left( U(y)-U^{\lambda }(y) \right)\ge \frac{2c_1(n) \lda^{1-n}}{5}$ for $y$ with
$|y|=\lda$.
Since $v_k(y) \to U(y)$ in $C^2_{loc}(\R^n)$ as $k\to \infty$,  for any fixed $R>>1$ we have
\be \label{eq:ms-start-1}
v_k(y)-v_k^{\lambda}(y)>\frac{c_1(n)}{5}(|y|-\lambda)|y|^{1-n}, \quad \lambda<|y|<R,
\ee
provided $k$ is sufficiently large. We also have
\[
 v_k^{\lambda}(y)\le (1-3\epsilon_0)|y|^{2-n},\quad |y|\ge R,
\]
where $\va_0>0$ is some constant.  By Proposition \ref{v-lb-2} to be proved later,
\[
 v_k(y)\ge (1-\epsilon_0)|y|^{2-n},\quad |y|\ge R.
\]
Thus
\be
v_k(y)-v_k^{\lambda}(y)>\begin{cases} \frac{c_1(n)}{5}(|y|-\lambda)|y|^{1-n}, \quad \lambda<|y|<R, \\
               2 \epsilon_0 |y|^{2-n},\quad R\le |y|\le Q_1 l_k^{1/2}.
               \end{cases}
               \ee
               Our estimate \eqref{fest} implies that
               for all sufficiently large $k$ we have
                $v_k(y)-v_k^{\lambda}(y)+f_{\lambda}(y)>0$ on $\lambda<|y|<Q_1 l_k^{1/2}$.
  Therefore, \eqref{v-k-s} follows.

The critical position in the moving sphere method is defined by
$$\bar \lambda:=\sup\{\lambda \le 1+\delta_1:
   v_k(y)-v_k^{\mu}(y)+f_{\mu}(y)>0, \quad \forall ~y\in
 \Sigma^k_{\mu}\setminus \{S_k\} \mbox{ and }  1-\delta_1 <\mu<\lambda \}.$$
By \eqref{v-k-s}, $\bar \lda$ is well-defined. In order to reach to the final contradiction we claim that
$\bar \lda=1+\delta_1$.

If $\bar \lambda<1+\delta_1$,  by the definition \eqref{eq:Q-1} of $Q_1$ and $|f_{\lda}|=o(1)M_k^{-1}$
on $\partial B_{Q_1 l_k^{1/2}}$,  we still have
$v_k-v_k^{\bar \lambda}+f_{\bar \lambda}>0$ on $\partial B_{Q_1 l_k^{1/2}}$. By the maximum principle,
$v_k-v_k^{\bar \lambda}+f_{\bar \lambda}$ is strictly positive in $ \Sigma^k_{\bar \lambda}$ and
$\frac{\pa }{\pa r}(v_k-v_k^{\bar \lambda}+f_{\bar \lambda})>0$ on $\partial B_{\bar \lambda}$,
therefore  we can move spheres a little further than $\bar \lambda$  by a standard argument in
the moving spheres method---the presence of a potential singularity of $v_k$ at $S_k$ does not
create any issues in applying this method, as was done in \cite{ChenLin3, Li, f-mar,x-z-1}.
This contradicts the definition of $\bar \lambda$. Therefore, the claim is proved.

Sending $k$ to $\infty$ in the inequality
\[
v_k(y)-v_k^{\bar \lda }(y)+f_{\bar \lda}(y)\ge 0 \quad \mbox{for }1+\delta_1=\bar \lda <|y|<Q_1 l_k^{1/2},
\]
we have
\[
U(y) \ge U^{\bar \lda }(y) \quad \mbox{for all }~ |y| \ge \bar \lda=  1+\delta_1,
\]
which is a clear violation of \eqref{eq:model-fact}. This contradiction concludes the proof of Theorem \ref{main-ub}.

\section{Blow up analysis for local solutions to the Yamabe equation}
\label{sec:blow-up}

In this section, we summarize a few key facts in Khuri-Marques-Schoen \cite{KMS} needed for our
analysis of the behavior of $u(x)$ near $x=x_k$ (namely for $u_k(z)$ near $z=0$) for the
$n\ge 7$ cases.

Suppose that $g_k$ are smooth metrics defined in $B_1$ satisfying
\be \label{eq:normal-conf-smth-1}
\|g_k\|_{C^{n+2}(B_1)}\le C, \quad \det g_k=1 \quad \mbox{in }B_1,
\ee
where $n\ge 3$ is the dimension and $k=1,2,\dots$, and $B_1$ is a normal coordinates chart of
$ g_k(x)=\exp (h_{ij}(x)),$
where we dropped the subscript  $k$ of $h_{ij}$, and in addition,
\[
\sum_jh_{ij}(x)x^j=0 \quad \mbox{and}\quad \mathrm{trace}(h_{ij}(x))=0.
\]
Define, when $n\ge 6$,
\[
H_{ij}(x)= \sum_{2\le |\al|\le n-4} h_{ij \al} x^\al,
\]
\[
H_{ij}^{(l)}(x)= \sum_{ |\al|=l} h_{ij \al} x^\al, \quad |H_{ij}^{(l)}|^2= \sum_{ |\al|=l} |h_{ij \al}|^2,
\]
where $\al=(\al_1,\dots, \al_n)$, $\al_i\ge 0$ are integers, $|\al|=\al_1+\dots+\al_n$
and
$$
h_{ij \al}=\frac{\pa ^{\al} h_{ij}(0)}{\al !}=
\frac{ \pa_{x^1}^{\al_1} \pa_{x^2}^{\al_2} \dots  \pa_{x^n}^{\al_n }h_{ij}(0)}
{\al_{1}!\cdots \al_{n}!}. $$
Then $H_{ij}(x)=H_{ji}(x)$, $H_{ij}(x)x^j=0$ and $trac (H_{ij}(x))=0$. By the Taylor expansion,
\begin{equation}
\begin{split}
&\left|R_{g_k}-\partial_i \partial_j h_{ij}+\partial_l(H_{ij} \partial_l H_{il}) -\frac{1}{2}\partial _j H_{ij}\pa_l H_{il} +\frac14\pa_l H_{ij} \pa_l H_{ij} \right|\\
&\le C \sum_{|\al|=2}^d \sum_{i,j} |h_{ij \al}|^2|x|^{2|\al|} +C|x|^{n-2}
\end{split}
\end{equation}
and
\begin{equation} \label{eq:slc-expn-1}
\begin{split}
&\left|R_{g_k}-\partial_i \partial_j h_{ij} \right|\\&
\le C \sum_{|\al|=2}^d  \sum_{i,j} |h_{ij \al}|^2|x|^{2|\al|-2} +C|x|^{n-2},
\end{split}
\end{equation}
where $d=[\frac{n-2}{2}]$, see Proposition 4.3 of  Khuri-Marques-Schoen \cite{KMS} (see also Ambrosetti-Malchiodi \cite{AM},  Brendle \cite{Brendle}).  For $\va>0$, let
\[
\tilde H_{ij}(y)= H_{ij}(\va y).
\]
If $n\ge 8$, it was proved in  section 4 of \cite{KMS} that there exists a solution $\tilde Z_\va(y)$ of
\be \label{eq:corr-1}
\Delta \tilde Z_\va(y) +n(n+2) U^{\frac{4}{n-2}}(y)
\tilde Z_\va(y) =c(n ) \sum_{l=4}^{n-4}\sum_{i,j} \partial_i \partial_j \tilde H_{ij}^{(l)}(y) U(y),
\ee
satisfying that  $\tilde Z_\va(0)=0$, $\nabla \tilde Z_\va(0)=0$,
\[
\int_{\pa B_r} \tilde Z_\va(y)\,\ud S =\int_{\pa B_r}\tilde Z_\va(y) y^i \,\ud S= 0, \quad r>0, ~i=1,\dots, n
\]
and
\be \label{eq:corr-2}
|\nabla^ m \tilde Z_\va(y) |\le C \sum_{|\al|=4}^{n-4}
 \sum_{i,j} \va^{|\al|} |h_{ij \al}|(1+|y|)^{|\al|+2- n-m},
\ee
where $U(y)=(1+|y|^2)^{-\frac{n-2}{2}}$, $c(n)=\frac { (n-2) }{ 4(n-1) }$,
$C>0 $ is independent of $\va$ and $H_{ij}$, and $m=0,1,2$.
Note that if $Z_\va(x)= \va^{-\frac{n-2}{2}} \tilde Z_\va(\frac{x}{\va})$ and
$U_{\va}(x)= \va^{\frac{n-2}{2}} (\va^2+|x|^2)^{-\frac{n-2}{2}}$, we have
\be \label{eq:corr-3}
\Delta  Z_\va (x)+n(n+2) U_{\va}^{\frac{4}{n-2}}(x)
 Z_\va(x) =c(n ) \sum_{l=4}^{n-4}\sum_{i,j} \partial_i \partial_j H_{ij}^{(l)}(x) U_{\va}(x),
\ee
 and
 \be \label{eq:corr-4}
|\nabla^ m  Z_\va(x) |\le C  \va^{\frac{n-2}{2}}  \sum_{|\al|=4}^{n-4}
 \sum_{i,j} |h_{ij \al}|(\va+|x|)^{|\al|+2- n-m},
\ee
where $C>0 $ is independent of $\va$ and $H_{ij}$, and $m=0,1,2$.

Suppose that $\{u_k\}_{k=1}^\infty$ is a sequence of solutions of
\be \label{eq:local-smth}
-L_{g_k} u_k = n(n-2) u_k^{\frac{n+2}{n-2}} \quad \mbox{in }B_1, \quad u_k>0
\ee
with $g_k$ satisfying \eqref{eq:normal-conf-smth-1}.

We say $0$ is an isolated blow up point of $u_k$ if
$\lim_{k\to \infty}u_k(0)= \infty$, $0$ is a local maximum point of $u_k$, and
\[
u_k(x)\le A_1 |x|^{-\frac{n-2}{2}} \quad \mbox{in }B_{\rho_0},
\]
where $A_1, \rho_0$ are positive constants independent of $k$.

 We say $0$ is an isolated simple blow up point of $u_k$, if $0$ is an isolated blow up point and
\[
r^{\frac{n-2}{2}} \bar u_k(r) \mbox{ has exactly one critical point in }(0,\rho)
\]
for some constant  $\rho\in (0,\rho_0]$ independent of $k$, where
\[
\bar u_k(r)= \dashint_{\pa B_{r}} u_k\,\ud S.
\]

Define $\va_k=u_k(0)^{-\frac{2}{n-2}}$ and
\[
v_k(y)= \va_k^{\frac{n-2}{2}} u_k (\va_k y) \quad \mbox{for }y\in B_{1/\va_k}.
\]

In the following proposition, we take $h_{ij \al}=0$ and $ \tilde Z_{\va_k}=0$ when $n\le 5$.

\begin{prop}\label{prop:known} Let  $0$ be an isolated simple blow up point of $u_k$. Then,
for $|y|\le \rho \va_k^{-1}$,
\begin{align*}
\left|\nabla^m (v_k-U- \tilde Z_{\va_k})(y)\right| \le &C \sum_{|\al|=2}^{d-1}
\sum_{i,j} |h_{ij \al}|^2 \va_k^{2|\al|} |\ln \va_k|^{\theta_{|\al|}} (1+|y|)^{2|\al|+2-n-m}\\
&   + C \va_k^{n-3} (1+|y|)^{-1-m}, \quad \text{for $m=0,1,2$,}
\end{align*}
where $C>0$ depends only on the upper bound of $\|g_k\|_{C^{n+2}(B_1)}$, $A_1 $ and $ \rho_0$, $\theta_{|\al|}=1$ if $|\al|=\frac{n-2}{2}$ while $\theta_{|\al|}=0$ otherwise.

If $6\le n\le 24$, then
\be \label{eq:weyl-vanishing}
\sum_{|\al|=2}^d \sum_{i,j} |h_{ij \al}|^2 \va_k^{2|\al|} |\ln \va_k|^{\theta_{|\al|}} \le C \va_k^{n-2},
\ee
where $C>0$ depends only on the upper bound of $\|g_k\|_{C^{n+2}(B_1)}$, $A_1 $ and $ \rho_0$.
\end{prop}

\begin{proof} The first part can be found in \cite{LZhu99, D, f-mar-1, LZ2} for $n\le 5$, and in \cite{KMS} for $n\ge 6$. The second part can be found in the proof of Theorem 6.1 of  \cite{KMS}.

\end{proof}

The dimension restriction $n\le 24$  for \eqref{eq:weyl-vanishing} of
Proposition \ref{prop:known} is necessary to guarantee
 a positive lower bound of the Pohozaev quadratic form; see  Theorem  A.4 and Theorem A.8
 of \cite{KMS}.
We note that when \eqref{eq:weyl-vanishing} holds for $u_{k}$ and $g_{k}$, then
a corresponding	version holds for
the rescaled $v_{k}(y)=\tau_{k}^{\frac{n-2}{2}}u_{k}(\tau_{k}y)$ and
 $\hat g_{k}(y)=g_{k}(\tau_{k}y)$  for $\tau_{k}\to 0$ such that $v_{k}(0)\to \infty$,
 as in the set up for the proof
 of  Lemma 8.2 of  \cite{KMS}, with $\epsilon_{k}$ replaced by
 $v_{k}(0)^{-\frac{2}{n-2}}$,  thus Lemma 8.2 of  \cite{KMS} continues to hold
 without the dimension restriction $n\le 24$, as long as \eqref{eq:weyl-vanishing} holds.
 Therefore we have

\begin{prop}\label{prop:isol-to-isol-sim} Let $0$ be an isolated blow up point of $u_k$. Suppose that either $ n\le 24$ or \eqref{eq:weyl-vanishing} holds for some constant $C$ independent of $k$.
Then $0$ is an isolated simple blow up point of $u_k$, and
 \[
\left|\nabla^m (v_k-U- \tilde Z_{\va_k})(y)\right| \le C \va_k^{n-3} (1+|y|)^{-1-m}
\]
for every $|y|\le \rho \va_k^{-1}$ and $m=0,1,2$, where $C>0$ is independent of $k$.
The latter estimate is equivalent to
\be \label{eq:blow-up-expansion}
\left|\nabla^m (u_k-U_{\va_k}-  Z_{\va_k})(x)\right| \le  C \va_k^{\frac{n-2}{2}}(\va_k +|x|)^{-1-m}
\ee
for $|x|\le \rho$.

\end{prop}

\begin{proof} It suffices to show that  $0$ is an isolated simple blow up point of $u_k$, since the estimates will follow from Proposition  \ref{prop:known} and \eqref{eq:weyl-vanishing}.
If $n\le 24$, this is what Lemma 8.2 of \cite{KMS} asserts.  If $n\ge 25$ and  \eqref{eq:weyl-vanishing} is assumed, as discussed above,
the conclusion of Theorem 7.1 of  \cite{KMS} (the local sign restriction of Pohozaev integral) still holds and  the same proof of Lemma 8.2 of \cite{KMS} applies.

Therefore, we complete the proof.
\end{proof}

The reason we are willing to assume \eqref{eq:weyl-vanishing} in Proposition~\ref{prop:isol-to-isol-sim}
for $n\ge 25$ is that, in the setting of
 \eqref{eq:contra-hy-1},  assumptions  \eqref{eq:inner-flat}
 and \eqref{eq:25-wk-bd-1}  imply that
 a relevant version of \eqref{eq:weyl-vanishing} holds for appropriately rescaled $u$;
 see \eqref{eq:weyl-vanishing-3}.

We formulate a version of
 \eqref{eq:weyl-vanishing} for the more general situation:
For any local maximum point of $u_{k}$
 in $\bar x\in B_{1/2}$ with $u_k(\bar x)\ge 1$, there exists
a conformal  normal coordinates system centered at $\bar x$ such that
\be  \label{eq:weyl-vanishing-2}
\sum_{|\al|=2}^{d} \sum_{i, j} |\pa^\al (g_k) _{ij}(0)|^2 \va_{\bar x, k}^{2|\al|}
 |\ln \va_{\bar x, k} |^{\theta_{|\al|}} \le  C \va_{\bar x, k}^{n-2}
\ee
for some constant $C$ independent of $k$,
where $\va_{\bar x, k}=u_k(\bar x)^{-\frac{2}{n-2}}$.

\begin{prop}\label{prop:KMS-8.2} Suppose that $0$ is a local maximum point of $u_k$
and $\lim_{k\to \infty}u_k(0)=\infty$. If $n\ge 25$, suppose further that for any
local maximum point of $u_{k}$  in $\bar x\in B_{1/2}$ with $u_k(\bar x)\ge 1$,
there exists  a conformal  normal coordinates system centered at $\bar x$
such that \eqref{eq:weyl-vanishing-2} holds.  Then
$0$ is an isolated simple blow up point of $u_k$ in some ball $B_{\rho}$ with $0<\rho<1$.
\end{prop}
\begin{proof} If $n\le 24$,  the proof  amounts to localizing the arguments in  section 8 of \cite{KMS}
on compact manifolds. There are only two places which need some modification.
 First, by considering the function $(1/2- |x|)^{\frac{n-2}{2}} u_k(x)$ in $\bar B_{1/2}$,
  we can obtain a local bubbles decomposition in $B_{1/2}$ to replace
  Proposition 8.1 of \cite{KMS}. In fact, this was done by Lemma 1.1 of Han-Li \cite{HanLi}.
  Second, to prove a positive lower bound for the distance between centers of bubbles
  (Proposition 8.3 of \cite{KMS}),
  we can use a selecting process to find two almost closest
  two bubbles and scale them apart; see  Lemma 2.1 of Niu-Peng-Xiong \cite{NPX}
  or the proof of Proposition 8.2 of Almaraz \cite{Al}.   The rest of section 8 of \cite{KMS},
  in particular Lemma 8.2,
  can be applied identically and Proposition \ref{prop:KMS-8.2} follows.

  If $n\ge 25$, \eqref{eq:weyl-vanishing-2} implies \eqref{eq:weyl-vanishing}.   Given Proposition \ref{prop:isol-to-isol-sim},  the proof is similar as above.
\end{proof}

We next formulate and prove the properties of the conformal normal coordinates in our set up.

\begin{lem} \label{lem:conf-normal} Let $g$ be a smooth Riemannian metric defined in $B_1\setminus \{0\}$ and satisfy \eqref{eq:inner-flat}. Then there exist constants $0<\sigma<\bar \sigma<1/16<\Lambda$, depending only $n,g$ and $C_0$ in \eqref{eq:inner-flat}, such that for any  $\bar z\in B(0,\frac 1{10})\setminus \{0\}$ one can find a function $\kappa_{\bar z} \in C^\infty(B_1)$ satisfying
\be \label{eq:conformal-factor}
\frac 1{\Lambda}\le \kappa_{\bar z} \le \Lambda,\quad \kappa_{\bar z} (\bar z)=1,\quad |\nabla \kappa_{\bar z} (\bar z)|=0
\ee
such that the conformal metric $\varkappa=\kappa_{\bar z}^{-\frac{4}{n-2}}g$ has the following properties. There exists a smooth bijection $\phi_{\bar z}:B_{1/2}\to B_{1/2}+\{\bar z\}$ satisfying
\begin{itemize}
\item[(i)]  $\phi_{\bar z}(0)=\bar z$, $\nabla \phi_{\bar z}(0)$ is the identity matrix,
$\phi_{\bar z}(-\bar z)=0$,
\[
\Lda^{-1}\le |\nabla \phi_{\bar z}| \le \Lda\quad \mbox{and}\quad
|\nabla^m \phi_{\bar z}| \le \Lda|\bar z|^{-(m-1)} \quad \mbox{in } B_{1/2}\]
 for $m=2,\dots, n$;
\item[(ii)] In this coordinates  system $(B_{1/2}, \phi_{\bar z})$, write
$\varkappa(\phi_{\bar z}(x))=\varkappa_{ij}(x)\ud x^i \ud x^j$. We have
\[
\det \varkappa_{ij}(x)=1  \quad \mbox{for }|x|\le \sigma |\bar z|
\]
and for $m=0,1,2,\dots, n$
\be \label{eq:varkappa}
\sum_{i,j=1}^n | \nabla^{m} (\varkappa_{ij}(x)- \delta_{ij}) | \le \begin{cases}
\Lda |\bar z|^{\tau-m} (\frac{|x|}{|\bar z|})^{\max\{2-m,0\}} &\quad \mbox{if }|x|\le 2\sigma|\bar z|, \\
\Lda |x+\bar z|^{\tau-m} &\quad \mbox{if }|x|> 2\sigma|\bar z|.
\end{cases}
\ee

\end{itemize}

\end{lem}

\begin{proof} For each $\bar z\in B_{1/10}\setminus \{0\}$, we set $\bar r=|\bar z|$ and define
\[
g^{\bar r}(y):=g_{ij}^{\bar r}(y) \ud y^i \ud y^j \quad \mbox{for } |y|<1/2,
\]
where $g_{ij}^{\bar r}(y) =g_{ij}(\bar z+\bar ry) $.
By \eqref{eq:inner-flat}, we have, for $|y|<1/2$,
\be \label{eq:lem1-1}
| \nabla^l( g^{\bar r}_{ij}(y)-\delta_{ij})|\le C\bar r^{\tau}\quad \mbox{for } l=0,1,\dots, n+2.
\ee
Hence, there exists a constant $\delta_0>0$ independent of $\bar z$ such that
the exponential maps $\exp_0^{g^{\bar r}}(x)$ centered at $0$ of $(B_{1/2}, g^{\bar r}(y))$
 is well defined for $|x| \le \delta_0$, namely,
$x$, $|x|<\delta_0$, provides a geodesic normal coordinates
for $y=\exp_0^{g^{\bar r}}(x)$ in the metric $g^{\bar r}(y)$.

By G\"unther \cite{G}, there exist a positive function $\kappa (y)\in C^\infty(B_{1/4})$
and $\delta_1<\frac{\delta_0}{2}$ such that the metric
$h(y)=\kappa^{-\frac{4}{n-2}}(y) g^{\bar r}(y)$, when expressed in terms of $x$ via
$y=\exp_0^{h}(x)$, satisfies
\be \label{eq:conf-norm-1}
\det  (h_{ij} (\exp_0^{ h} (x)))= 1 \quad \mbox{for }|x|<2\delta_1.
\ee
 Moreover, $\kappa(\exp_0^{ h}(0))=1$,  $\nabla \kappa(\exp_0^{ h}(0))=0$,
\[
\Lda^{-1} \le \kappa(\exp_0^{ h}(x) ) \le \Lda, \quad |\nabla ^m \kappa(\exp_0^{ h}(x))|
\le \Lda \bar r^\tau \quad \mbox{for }|x|\le 4\delta_1, ~m=1,\dots, n,
\]
where $\Lda$ and $\delta_1$ depend only on $C_0$ and $n$. Since
\[
h_{ij}(y)= \sum_{k,l=1}^n \frac{\partial x_k}{\partial y_i}
 h(\frac{\partial}{\partial x_k}, \frac{\partial}{\partial x_l})
\frac{\partial x_l}{\partial y_j},
\]
and $h(\frac{\partial}{\partial x_k}, \frac{\partial}{\partial x_l})=\delta_{kl}+O(|x|^2)$ near $x=0$,
it follows from \eqref{eq:lem1-1} that
\[
\left| (\frac{\pa y}{\pa x})\Big|_{x=0} -I\right|\le \Lda \bar r^{\tau},
\] where $ (\frac{\pa y}{\pa x})\Big|_{x=0}$ is the Jacobi matrix at $x=0$ and
$I$ is the identity matrix. Hence, we can find $\bar \sigma>0 $ such that for $\bar r<\bar \sigma$,
there exists a  smooth bijection map $\phi$ from $\R^n\to \R^n$ extending
$y= \exp_{0}^{ h} (x)$ for $|x|<2\delta_1$ such that
\[
\phi (x)=x \quad \mbox{for }|x|>4\delta_1.
\]
Set $\sigma:=2\delta_1$ and
\be \label{eq:the-map}
\phi_{\bar z}(x):= \bar z+\bar r \phi(\frac{x}{\bar r}) \quad \mbox{for }x\in B_{1/2}.
\ee

We can extend and modify  $\kappa(y)= \kappa(\exp_0^{h}(x))$
for $2\delta_1 \le |x|\le 4\delta_1$ so that
\[
\kappa(\phi(x))= \kappa(0) \quad \mbox{for }|x|\ge 4\delta_1,\] and
\[
(C\Lda)^{-1} \le \kappa(\phi(x)) \le C \Lda, \quad |\nabla ^m \kappa(\phi(x))|
\le C \Lda  |\bar z|^\tau \quad \mbox{for }|x|\in \R^n, ~m=1,\dots, n.
\]
Set
\be \label{eq:kappa-z}
\kappa_{\bar z}(\phi_{\bar z}(x)):= \kappa(\phi(\frac{x}{|\bar z|})).
\ee
Then
\[
\varkappa_{ij}(x)=\kappa(\phi(\frac{x}{\bar r}))^{-\frac{4}{n-2}}
\sum_{k, l=1}^n \frac{\partial \phi_k}{\partial x_i}  \left(\frac{x}{\bar r}\right)
g_{kl}\left(\bar z+ \bar r \phi(\frac{x}{\bar r})\right) \frac{\partial \phi_l}{\partial x_j}
\left(\frac{x}{\bar r}\right),
\]
and it is easy to check that $\kappa_{\bar z}(z) $, $\varkappa_{ij}(x)$,
and $ \phi_{\bar z}(x)$ satisfy all the conclusions in the lemma. Therefore, we complete the proof.

\end{proof}

\section{Details in proving the upper bound in Theorem \ref{main-ub}}
\label{sec:4}

We now furnish details for the construction of $f_\lda$ satisfying \eqref{fest} and \eqref{fsuper}.
The construction has some differences between the $3\le n\le 6$ and $n\ge 7$ cases.
We first summarize lower bounds of $u_k$ and $v_k$
in Propositions \ref{lem:v-lb} and \ref{v-lb-2}, and then an upper bound of
$|x_k|$ in terms of $l_k^{1/2}$ in Proposition \ref{prop:x-k-estimates}.

By Lemma \ref{lem:conf-normal}, $g$ is at least H\"older continuous in $B_{1/2}$ and
$|R_{g}(z)| \le C |z|^{\tau-2}$.
 By Hardy inequality,  there exists a $\delta>0$ such that
\be \label{eq:positive-first-eigen}
\int_{B_{\delta }} \Big(|\nabla_{g} \phi|^2- c(n)|R_{g}|\varphi^2\Big)\,\ud vol_{g} \ge \int_{B_\delta} |\varphi|^2 \,\ud vol_{g}, \quad \forall~\varphi\in H_0^1(B_\delta).
\ee  Without loss of generality, we assume $\delta=1/2$.

\begin{prop}\label{lem:v-lb}
There exists $c_0>0$ independent of $k$ such that $u_k(z)>c_0$ in $B_{1/2}$.
\end{prop}

\begin{proof}  Since the conformal factors $\kappa_{x_k}$ are uniformly controlled, it suffices to
prove that $u(x) \ge c>0$ on $B_{1/2}\setminus\{0\}$ for some $c>0$.
Since $u(x)$ is a positive solution of
 \eqref{eq:main-u1} in $B_{1/2}\setminus\{0\}$, it is a positive supersolution of $L_{g}$ there.
 The Hardy inequality for $g$ on $B_{\delta}$ (we have taken $\delta=1/2$)
 implies that there exists a classical solution $v(x)$
 on $B_{1/2}$ of $L_{g}v(x)=0$ with $v(x)=u(x)$ on
 $\partial B_{1/2}$, and that $v(x)>0$ in $B_{1/2}$---we have used
 the De Giorgi-Nash-Moser theory here
 as explained in the introduction. The maximum principle holds for $L_{g}$ on
 $B_{1/2}$, and just as in proving Bocher's theorem for harmonic functions, we conclude that $u(x)\ge v(x)$
 on $B_{1/2}\setminus\{0\}$, it then follows that $u(x)\ge \min_{\overline{B}_{1/2}} v>0$.
\end{proof}

\begin{prop}\label{v-lb-2}
For any given $\epsilon_0>0$, there exists $R>0$ such that for all sufficiently large $k$
\be \label{eq:lb-green}
v_k(y)\ge (1-\epsilon_0)|y|^{2-n}, \quad R<|y|<\frac{l_k}{2}.
\ee
\end{prop}

\begin{proof} By Proposition \ref{lem:v-lb}, we have  $v_k(y)\ge c_0/M_k$.
Hence,  \eqref{eq:lb-green} holds if  $|y|\ge l_k^{3/4}=M_k^{\frac{3}{2(n-2)}}$.

Next, we consider  $|y|<l_k^{3/4}$.
Since $v_k \to U$ in $C^2_{loc}(\R^n)$ as $k\to \infty$,
for any $\epsilon_0>0$ small and $R$ large,  we have
 \begin{equation}\label{on-R}
v_k(y)\ge (1-\frac{\epsilon_0}{8})(1+|y|)^{2-n}, \quad |y|\le R,
\end{equation}
when $k$ is large.
Let $G_k\in C^2(B_{1/2}\setminus \{0\})$ be a nonnegative solution of
\[
-L_{g_k}G_k =0 \quad \mbox{in }B_{1/2}\setminus \{0\}, \quad G_k=0 \quad \mbox{on }\pa B_{1/2}
\]
and
\[
\lim_{z\to 0} |z|^{n-2} G_k(z)= 1.
\]
Making use of the standard local estimates of linear elliptic equations, we have
\be \label{eq:geenfunct-est-1}
G_k(z)= |z|^{2-n}+ a^k(z),
\ee
where
\[
|a^k(z)| \le C|z|^{2-n+\gamma}
\]
for some constants $C>0$ and $0<\gamma \le 1$ independent of $k$. By \eqref{on-R}
and \eqref{eq:geenfunct-est-1}, we have $u_k\ge (1-\frac{\epsilon_0}{4})  M_k^{-1} G_k$
on $\pa B_{R l_k^{-1}}$ when $k$ is large.  Applying the maximum principle to
$u_k-(1-\frac{\epsilon_0}{4}) M_k^{-1}G_k$, we obtain
\[
u_k(z) -(1-\frac{\epsilon_0}{4} ) M_k^{-1}G_k(z) \ge 0 \quad
\mbox{in }B_{1/2}\setminus B_{R l_k^{-1}},
\]
where we have used $-L_{g_k}$ is coercive in $H_0^1(B_{1/2})$, i.e.,
\eqref{eq:positive-first-eigen}. It follows that
\[
v_k(y) \ge (1-\frac{\epsilon_0}{4} ) |y|^{2-n} \Big(1-C (l_k^{-1}|y|)^\gamma \Big) \quad
\mbox{for }R<|y|<\frac12 l_k.
\]
Hence, if  further $|y|<l_k^{3/4}$,  we have
\begin{equation}\label{v-low-b-2}
v_k(y)\ge (1-\frac{\epsilon_0}{4} ) (1-C l_{k}^{-\frac{\gamma}{4}}) |y|^{2-n}\ge
(1-\epsilon_0)|y|^{2-n},
\end{equation}
provided $k$ is sufficiently large. This completes our proof.
\end{proof}

\begin{prop}\label{prop:x-k-estimates} Under the assumptions of Theorem \ref{main-ub} and Lemma \ref{lem:blow-up},
we have
\[
|x_k|\le \bar C l_k^{-1/2}
\]
for some $\bar C>0$ independent of $k$.
\end{prop}

\begin{proof}  If $n\ge 25$, the proposition follows immediately from the assumption
\eqref{eq:25-wk-bd-1}. Now, we assume $n\le 24$, and define
$\tilde u_k(y)= |x_k|^{\frac{n-2}{2}} u_k(|x_k| y)$ and
$\tilde g_k(y)= g_{k}(|x_k| y)$.  Then $\tilde u_k(y)$ is a smooth solution of
\[
-L_{\tilde g_k} \tilde u_k= n(n-2) \tilde u_k^{\frac{n+2}{n-2}} \quad \mbox{in }B_1,
\]
and $0$ is a local maximum point of $\tilde u_k$,
$\tilde u_k(0)=|x_k|^{\frac{n-2}{2}}M_k\to \infty$ as $k\to \infty$,
and
\[
\det \tilde g_k =1\quad \mbox{in }B_\sigma, \quad \|\tilde g_k\|_{C^{n+2}(B_1)}\le C
\]
for some $C>0$ independent of $k$.  By Proposition \ref{prop:KMS-8.2},
$0$ must be an isolated simple blow up point of $\tilde u_k$ for some $0<\rho<\sigma$.
 As a consequence of the last estimate of Proposition~\ref{prop:isol-to-isol-sim},
$\tilde u_k(0) \tilde u_k(y) \le C$ for $|y|=\rho$
for some $C>0$ depending on $\rho$ and all sufficiently large $k$.
However, for $|y|=\rho$,
\[
\tilde u_k(0) \tilde u_k(y) =|x_k|^{\frac{n-2}{2}}M_k |x_k|^{\frac{n-2}{2}} u_k(|x_k| y)
=|x_k|^{n-2} M_k u_k(|x_k| y)\ge c_0 |x_k|^{n-2} M_k
\]
where $c_0>0$ is the constant in Proposition \ref{lem:v-lb}. So we obtain an upper bound for
$|x_k| \le \bar C l_k^{\frac 12}$.

\end{proof}

Let $Q_1>0$ such that for all $\lda\le 2$
\be \label{eq:Q-1}
2U^{\lda}(y)<\frac{c_0}2 M_k^{-1} \quad \mbox{for }|y|\ge \frac{Q_1 l_k^{1/2}}{4},
\ee
where $c_0>0$ is the constant in Proposition \ref{lem:v-lb}. This choice of $Q_1$
can guarantee that $v_k^\lda<v_k$ near the boundary of $B_{Q_1 l_k^{1/2}}$
due to \eqref{eq:lb-green}.

\begin{rem}\label{rem:quantities}
We now summarize the relations among  $|x_k|, l_k^{1/2}$ and $|S_k|$ for future references:
\begin{itemize}
\item[(i).] $|x_k|\to 0$, $M_k\to \infty$, $|S_k|= l_k |x_k|\to \infty$ as $k\to \infty$,
\item[(ii).] $|x_k|\le \bar C  l_k^{-1/2}$, $|S_k|\le \bar Cl_k^{1/2} < Q_1 l_k^{1/2}$,
\item[(iii).] $l_k^{\tau}=M_k$, $l_k^{-2}|x_k|^{\tau-2}=l_k^{-\tau } |S_k|^{\tau -2}
= M_k^{-1} |S_k|^{\frac{n-6}{2}} $ and
$l_k^{-4}|x_k|^{\tau-4}= l_k^{-\tau } |S_k|^{\tau -4}=M_k^{-1} |S_k|^{\frac{n-10}{2}}  $,
using $\tau=\frac{n-2}{2}$.
\end{itemize}
\end{rem}

\subsection{Case of $3\le n\le 6$. }

Here we first provide an upper bound for $E_{\lda}$ as defined in \eqref{Elda}
in terms of some powers of $|y|$ and $|y-S_{k}|$,
then construct $f_{\lda}(y)$ with respect to these power functions in Lemmas \ref{lem:h-est-1}
and \ref{lem:h-est-2}.

Let $\chi_k\in C_c^\infty (B_{|S_k|/2}(S_k))$ be a cutoff function satisfying $0\le \chi_k\le 1$
and $\chi_k=1$ in $B_{|S_k|/4}(S_k)$. Let
\be
\sigma_k:=\|v_k-U\|_{C^2(B_2)}.
\ee
We have  $\sigma_k \to 0$ as $k\to \infty$. Note also that $\frac 12 \le \tau\le 2$ here.

\begin{prop}\label{e-e-lam}  Suppose $n\le 6$. For $\lda \in [1/2,2]$ and
$\lda\le |y|\le Q_1l_k^{1/2}$,  we have
\begin{equation}
E_\lda (y)\le C_1\Big(E_\lda^{(1)}(|y|) +E_\lda^{(2)}(y)\Big),
 \end{equation}
where  $C_1>0$ is independent of $k$,
\[
E_\lda^{(1)}(|y|) =\begin{cases}
  l_k^{-\tau} |S_k|^{\tau-4}|y|^{4-n} +\sigma_k l_k^{-\tau}|S_k|^{\tau-2}|y|^{-n},& \quad
  |y|< \sigma |S_k|,\\[2mm]  l_k^{-\tau}  |y|^{\tau-n} ,&\quad |y|\ge \sigma |S_k|,
 \end{cases}
\]
and
\[
E_\lda^{(2)}(y) = l_k^{-\tau}|S_k|^{2-n}|y-S_k|^{\tau-2}\chi_k(y).
 \]
\end{prop}

\begin{proof}  If $|y|<\sigma |S_k|$, the proof is identical to that of Proposition 2.3 of \cite{x-z-1};
see also Proposition 2.1 of \cite{LZ2}. We include the proof here for reader's convenience.
We start from the second term of $E_{\lambda}$:
$$I:=(\bar b_j(y)\partial_j v_k^{\lambda}(y)+\bar d_{ij}(y)\partial_{ij}v^{\lambda}_k(y)).$$
Since $y$ is conformal normal for $g_k(y)$ in $|y|\le \sigma |S_k|$, we have
\begin{equation}
0=(\Delta_{g_k}-\Delta ) V(y) = (\bar b_j(y)\partial _j+\bar d_{ij}(y)\partial_{ij})V(y)
\label{rrr}
\end{equation}
for any smooth radial function $V(y)=V(|y|)$ and $|y|\le \sigma |S_k|$. It follows that,
for $|y|\le \sigma |S_k|$,
$$
I=(\bar b_j(y)\partial_j+\bar d_{ij}(y)\partial_{ij})
[ (v_k-U)^\lambda(y)].
$$
By a direct computation,
$$
\partial_j
\bigg\{  (\frac \lambda{|y|})^{n-2}  (v_k-U)(y^\lambda)\bigg\}
=
\partial_j
\bigg\{  (\frac \lambda{|y|})^{n-2} \bigg\}  (v_k-U)(y^\lambda)
+
  (\frac \lambda{|y|})^{n-2}
\partial_j \bigg\{   (v_k-U)(y^\lambda)
\bigg\},
$$
\begin{align*}
&\partial_{ij}\bigg\{  (\frac \lambda{|y|})^{n-2}  (v_k-U)(y^\lambda)\bigg\}
\\&=\partial_{ij}\bigg\{  (\frac \lambda{|y|})^{n-2} \bigg\}  (v_k-U)(y^\lambda)+
\partial_i \bigg\{  (\frac \lambda{|y|})^{n-2} \bigg\}
\partial_j  \bigg\{    (v_k-U)(y^\lambda)  \bigg\}\\ & \quad
+ \partial_j
\bigg\{  (\frac \lambda{|y|})^{n-2} \bigg\}
\partial_i
 \bigg\{    (v_k-U)(y^\lambda)  \bigg\}
+
(\frac \lambda{|y|})^{n-2}
\partial_{ij}
 \bigg\{    (v_k-U)(y^\lambda)  \bigg\}.
\end{align*}
Since  $\bar d_{ij}\equiv \bar d_{ji}$, using \eqref{rrr} with $V(y)=(\frac \lambda{|y|})^{n-2}$, we have
\begin{align*}
I=&
(\frac \lambda{|y|})^{n-2}
\bar b_j (y) \partial_j
 \bigg\{    (v_k-U)(y^\lambda)  \bigg\}
+2\bar d_{ij} (y) \partial_i
\bigg\{  (\frac \lambda{|y|})^{n-2} \bigg\}
\partial_j
 \bigg\{    (v_k-U)(y^\lambda)  \bigg\}
\\&+(\frac \lambda{|y|})^{n-2}
\bar d_{ij} (y) \partial_{ij}
 \bigg\{    (v_k-U)(y^\lambda)  \bigg\},
\end{align*}
for $|y|\le \sigma |S_k|$. To evaluate terms in $I$, we observe that for $z\in B_2$,
\begin{equation} \label{13-1}
\begin{split}
(v_k-U)(z)&=O(\sigma_k)|z|^2, \\
|\nabla_z (v_k-U)(z)|&=O(\sigma_k)|z|,  \\ |\nabla_z^2(v_k-U)(z)|&=O(\sigma_k),
\end{split}
\end{equation}
where we have used  $(v_k-U)(0)=|\nabla (v_k-U)(0)|=0$. Here we recall that $\sigma_k=\|v_k-U\|_{C^2(B_1)}\to 0$.
It follows from the first components of (\ref{bik}), (\ref{dijk}) and (\ref{ck}) that
\begin{equation} \label{eq:error-1}
\begin{split}
I&= O(|x_k|^{\tau-2}) \sigma_k   l_k^{-2} \Big(|y|^{2-n}  |y| |y^\lambda | |\nabla_y y^\lambda| +|y|^2|y|^{1-n} |y^{\lambda}|  |\nabla_y y^\lambda| \\& \quad +|y|^{2-n} |y|^2  (|y^\lambda | |\nabla_y^2 y^\lambda| +|\nabla_y y^\lambda| ^2 ) \Big) \\& = \sigma_k
  l_k^{-2} O(|x_k|^{\tau-2})|y|^{-n}=\sigma_k  l_k^{-\tau} O(|S_k|^{\tau-2})|y|^{-n}, \quad |y|<\sigma |S_k|.
\end{split}
\end{equation}
Similarly,
\begin{align*}
&(\frac{\lambda}{|y|})^{n+2}\left(\bar b_j(y^{\lambda})
\partial_j v_k(y^{\lambda})+
\bar d_{ij}(y^{\lambda})\partial_{ij}v_k(y^{\lambda})\right)\\
=& \sigma_k
 l_k^{-2} |y|^{-n} O(|x_k|^{\tau-2}) |y|^{-n}= \sigma_k
 l_k^{-\tau } O(|S_k|^{\tau-2}) |y|^{-n},\quad |y|<\sigma |S_k|
\end{align*}
and
\begin{align*}
& |\bar c(y)||v_k^\lambda(y)-U^\lambda(y)|
+(\frac \lambda{|y|})^{n+2}
|\bar c(y^\lambda)||v_k(y^\lambda)-U(y^\lambda)| \\
=& \sigma_k  l_k^{-4} O(|x_k|^{\tau-4}) |y|^{-n}=\sigma_k  l_k^{-\tau} O(|S_k|^{\tau-4}) |y|^{-n}
\quad \mbox{for}\quad |y|<\sigma |S_k|.
\end{align*}
Finally, using the estimates of $\bar c$  we have
\[
\bar c(y) U^\lda (y)- (\frac \lambda{|y|})^{n+2} \bar c(y^\lambda) U(y^\lambda)= l_k^{-4}  O(|x_k|^{\tau-4})
 |y|^{4-n}=l_k^{-\tau} O(|S_k|^{\tau-4}) |y|^{4-n}.
\]
This finishes the proof for the  case $|y|\le \sigma |S_k|$.

 If $|y|\ge \sigma |S_k|$, by the estimates in (\ref{bik}), (\ref{dijk}) and (\ref{ck}), we have
\begin{align*}
E_{\lambda}(y)&\le C l_k^{-\tau}|y-S_k|^{\tau-2}|y|^{2-n}+Cl_k^{-2}|x_k|^{\tau-2}|y|^{-2-n}\\&
=  C l_k^{-\tau}(|y-S_k|^{\tau-2}|y|^{2-n} +|S_k|^{\tau-2}|y|^{-2-n})\\&
\le C l_k^{-\tau}(|y-S_k|^{\tau-2}|S_k|^{2-n} \chi_k(y) +|y|^{\tau-n} ).
\end{align*}
This finishes our proof.
\end{proof}

We next construct supersolutions of the linearized operator \eqref{eq-w}. Since $E_\lda^{(1)}(|y|)$ is radial, we
 let $f_{\lda}^{(1)}(r)$  be the radial solution of
\be \label{eq:barrier-ode}
\begin{split}
\Delta  f_{\lda}^{(1)} =\frac{\ud^2}{\ud r^2 }f_{\lda}^{(1)} &+\frac{n-1}{r} \frac{\ud}{\ud r } f_{\lda}^{(1)}
 = -N E_\lda^{(1)}(r), \quad r\in (\lda, Q_1l_k^{1/2}), \\[2mm]
f_{\lda}^{(1)}(\lda)&=\frac{\ud}{\ud r }   f_{\lda}^{(1)}(\lda)=0,
\end{split}
\ee
where $N>2C_1+2$ is a constant.

\begin{lem}\label{lem:h-est-1} Suppose that $3\le n\le 6$. Let $Q_1$ be a constant defined in \eqref{eq:Q-1} and $\lambda\in [1/2,2]$. Then for  $\lambda<|y|<Q_1l_k^{1/2}$, there hold
\be  \label{eq:lem37a}
f_{\lda}^{(1)}(|y|)<0,\quad f_{\lda}^{(1)}(|y|)=o(1)M_k^{-1}
\ee and
\be \label{eq:lem37b}
\left[ \Delta +\bar b_i(y)\partial_i+\bar d_{ij}(y)\partial_{ij}-\bar c(y)\right] f_{\lda}^{(1)}(y)+(N-1)E_\lda^{(1)}(y)  \le E_\lda^{(2)}(y).
\ee
\end{lem}

\begin{proof} By solving the Cauchy problem of \eqref{eq:barrier-ode}, we obtain
\[
 f_{\lda}^{(1)}  (r)= - N\int_{\lda}^r \frac{1}{s^{n-1}} \left(\int_\lda^s t^{n-1} E_\lda^{(1)}(t)\,\ud t \right)\,\ud s.
\]
Since $E_\lda^{(1)} $ is positive, $f_\lda^{(1)}(r)<0$ for $r>\lda\ge \frac12$.
Note that for $p\in \R$ and $\frac12 \le \lda \le 2$,
\begin{align*}
\int_{\lda}^r \frac{1}{s^{n-1}} \left(\int_\lda^s t^{n-1} t^{p-n}\,\ud t \right)\,\ud s &=
\begin{cases} \frac{1}{p} \int_{\lda}^{r} \frac{s^p-\lda^p}{s^{n-1}}\,\ud s &\text{if $p\ne 0$,}\\
 \int_{\lda}^{r} s^{1-n} \ln \frac{s}{\lda}\, \ud s &\text{if $p=0$,}\\
 \end{cases}
\\[2mm] &
\le C \begin{cases}
1,& \quad \mbox{if }p<n-2,\\
\ln \frac{r}{\lda},& \quad \mbox{if }p=n-2, \\
r^{p+2-n},& \quad \mbox{if }p>n-2,
\end{cases}
\end{align*}
where $C>0$ depends only on $n$ and $p$.

By the expression of $E_\lda^{(1)}$ and using $p=4$ or $0$ as well as the relation
$|x_k|^{\tau}=|S_k|^{\frac{n-2}{2}} l_k^{-\frac{n-2}{2}}=|S_k|^{\frac{n-2}{2}} M_k^{-1}$ as recorded in
(iii) of  Remark \ref{rem:quantities},   we obtain, for $|y|\le \sigma |S_k|$ and when $3\le n \le 5$,
\begin{equation}\label{h-est}
\begin{split}
|f_{\lda}^{(1)}(y)|& \le
CN\left( l_k^{-4}|x_k|^{\tau-4} |y|^{6-n}+\sigma_kl_k^{-2}|x_k|^{\tau-2} \right)\\&
\le CN M_k^{-1}( |S_k|^{-\frac{n-2}{2}} +\sigma_k |S_k|^{-\frac{6-n}{2}}) \\&
=o(1) M_k^{-1}
 \quad \mbox{for }|y|\le \sigma |S_k|,
\end{split}
\end{equation}
where  $o(1)$ is with respect to $k\to \infty$; for the $n=6$ case, we only need to
modify the $|y|^{6-n}$ term in the first line into $\ln (\frac{|y|}{\lda})$ and the
$|S_k|^{-\frac{n-2}{2}}$ term in the second line into $|S_k|^{-2}\ln(\frac{|S_k|}{\lda})$.

For $|y|\ge \sigma |S_k|$ and $3\le n \le 6$---with the same modifications as above for
the $n=6$ case, we have
\be \label{eq:f-1-out}
\begin{split}
&|f_{\lda}^{(1)}(y)|\\&= |f_{\lda}^{(1)}(\sigma |S_k|)|  +N
\int_{\sigma |S_k|}^r\frac{1}{s^{n-1}}\left(\int_{\lambda}^{\sigma |S_k|}t^{n-1}|E_\lda^{(1)}(t)|\, \ud t+\int_{\sigma |S_k|}^st^{n-1}|E_\lda^{(1)}(t)|dt\right)\,\ud s \\ & \le
CN\left( l_k^{-4}|x_k|^{\tau-4} |S_k|^{6-n}+\sigma_kl_k^{-2}|x_k|^{\tau-2} \right) \\
& \quad + CN \Big(l_k^{-4} |x_k|^{\tau-4} |S_k|^{6-n}  +\sigma_kl_k^{-2}|x_k|^{\tau-2}|S_k|^{2-n} \ln|S_k| \Big) \\& \quad +CNl_k^{-\tau}|S_k|^{\tau+2-n}\\&
\le CN M_k^{-1} \Big( |S_k|^{-\frac{n-2}{2}} +\sigma_k |S_k|^{-\frac{6-n}{2}}  \Big)=o(1)M_k^{-1}   .
\end{split}
\ee
Hence, \eqref{eq:lem37a} is proved.

Next, we  shall prove \eqref{eq:lem37b}.

i). If $|y|<\sigma |S_k|$, we have
$\Delta_{\bar g_k}  f_{\lda}^{(1)}(y)= \Delta f_{\lda}^{(1)}(y)=-N E_\lda^{(1)}$ and
\begin{align*}
|\bar c(y) f_{\lda}^{(1)}(y)| &\le CN l_k^{-4}|x_k|^{\tau-4}|y|^2 \left( l_k^{-4}|x_k|^{\tau-4} |y|^{6-n}+\sigma_kl_k^{-2}|x_k|^{\tau-2} \right)\\
&= CN l_k^{-\tau} |S_k|^{\tau-4}|y|^{4-n}\left(l_k^{-\tau} |S_k|^{\tau-4}|y|^4+
\sigma_k l_k^{-\tau} |S_k|^{\tau-2} |y|^{n-2}\right).
\end{align*}
 Recall that
 \begin{align*}
 E_\lda^{(1)}(|y|)&= l_k^{-4} |x_k|^{\tau-4}|y|^{4-n} +\sigma_k l_k^{-2}|x_k|^{\tau-2}|y|^{-n}\\
 &= l_k^{-\tau} |S_k|^{\tau-4}|y|^{4-n}\left(1+\sigma_k |S_k|^2 |y|^{-4}\right)\\
 &\ge l_k^{-\tau} |S_k|^{\tau-4}|y|^{4-n}
 \end{align*}
  for $|y|<\sigma |S_k |$.  We shall show that $|\bar c(y)f_{\lda}^{(1)}(y)|=o(1) E_\lda^{(1)}(y)$.

First we note that for $|y|<\sigma |S_k |$,
\[
l_k^{-\tau} |S_k|^{\tau-4}|y|^4\le \sigma^4 l_k^{-\tau} |S_k|^{\tau}=\sigma^4 |x_k|^{\tau} \to 0.
\]
Furthermore,
\begin{align*}
\sigma_k l_k^{-\tau} |S_k|^{\tau-2} |y|^{n-2} &\le \sigma_k \sigma^{n-2} l_k^{-\tau} |S_k|^{\tau-2+n-2}\\
&\le \sigma_k \sigma^{n-2} Q_1^{\frac{3n-10}{2}} l_k^{-\tau+ \frac{3n-10}{4}}\\
& = \sigma_k \sigma^{n-2} Q_1^{\frac{3n-10}{2}} l_k^{\frac{n-6}{4}}\to 0,\\
\end{align*}
when $n\le 6$, where we have used $|y|\le |S_k|\le Q_1 l_k^{1/2}$.
Therefore, $|\bar c(y) f_{\lda}^{(1)}(y)|=o(1) E_\lda^{(1)}(y)$ for $|y|\le \sigma|S_k|$.
In conclusion,
\be \label{eq:37b-interior}
\left[ \Delta +\bar b_i(y)\partial_i+\bar d_{ij}(y)\partial_{ij}-\bar c(y)\right]f_{\lda}^{(1)}(y) = (-N+o(1)) E_{\lda}^{(1)} (y) \quad \mbox{for } |y|\le \sigma|S_k|.
\ee

ii). If $ \sigma |S_k| \le |y|\le Q_1 l_k^{1/2}$, then
$l_k^{-\tau}|S_k|^{2-n}\ge \sigma^{n-2}Q_{1}^{2-n}l_{k}^{2-n}$ and
$$E_{\lda}^{(2)}(y)= l_k^{-\tau}|S_k|^{2-n}|y-S_k|^{\tau-2}\chi_k
\ge \sigma^{n-2}Q_{1}^{2-n}l_{k}^{2-n} |y-S_k|^{\tau-2}
= \sigma^{n-2}Q_{1}^{2-n} M_{k}^{-2}|y-S_k|^{\tau-2}$$
when $|y-S_k|\le |S_{k}|/4$; when $|y-S_k|\ge |S_{k}|/4$, $|y-S_k|\ge |y|- |S_{k}|\ge |y|-4|y-S_k|$,
so $5 |y-S_k|\ge |y|$ and  due to $\tau\le 2$, we have
 $ |y-S_k|^{\tau-2}|/|y|^{\tau-n}\le 5^{2-\tau}|y|^{n-2}\le C l_k^{\frac{n-2}{2}}$ for some $C>1$,
and
\be \label{eq:E-lda-low}
\begin{split}
E_{\lambda}^{(1)}(|y|)+E_{\lda}^{(2)}(y)&= l_k^{-\tau}  |y|^{\tau-n}  + l_k^{-\tau}|S_k|^{2-n}|y-S_k|^{\tau-2}\chi_k \\&
\ge \frac{1}{C}l_k^{-\tau} |y-S_k|^{\tau-2}l_k^{\frac{2-n}{2}} =\frac{1}{C}M_k^{-2 } |y-S_k|^{\tau-2},
\end{split}
\ee
for some $C>1$---this certainly holds as well when $|y-S_k|\le |S_{k}|/4$.

In this region, we may no longer have $(\bar b_i(y)\partial_i+ d_{ij}(y)\partial_{ij})f_{\lda}^{(1)}(|y|)=0$,
but by the estimates  \eqref{bik}, \eqref{dijk} and \eqref{ck} of $\bar b_i, \bar d_{ij}$ and $\bar c$, we have
\begin{align*}
&|\bar b_i(y)\partial_i f_{\lda}^{(1)}(|y|)+d_{ij}(y)\partial_{ij}f_{\lda}^{(1)}(|y|) -\bar c(y) f_{\lda}^{(1)}(|y|)|\\&
\le Cl_k^{-\tau}\Big (|y-S_k|^{\tau-1} |\frac{\ud}{\ud r} f_{\lda}^{(1)}(|y|) | +
 |y-S_k|^{\tau} |\frac{\ud^2}{\ud r^2} f_{\lda}^{(1)}(|y|) | + |y-S_k|^{\tau-2} |f_{\lda}^{(1)}(|y|)|\Big)\\&
 =o(1) M_k^{-2} |y-S_k|^{\tau-2} = o(1) (E_{\lambda}^{(1)}(|y|)+E_{\lda}^{(2)}(y)),
\end{align*}
where we have used \eqref{eq:f-1-out},  and
\begin{align*}
|\frac{\ud}{\ud r} f_{\lda}^{(1)} (r) | = \frac{N}{r^{n-1}}\int_{\lda}^r t^{n-1} E_\lda^{(1)}(t)\,\ud t
\le CN M_k^{-1} r^{-\frac{n-2}{2}-1},
\end{align*}
\[
  |\frac{\ud^2}{\ud r^2} f_{\lda}^{(1)}(r) | \le C E^{(1)}_\lda(r)+\frac{n-1}{r}|\frac{\ud}{\ud r} f_{\lda}^{(1)} (r) |  \le CN M_k^{-1} r^{-\frac{n-2}{2}-2},
\]
when $r\ge \sigma |S_k|$.

It follows from \eqref{eq:E-lda-low} that,  for $ \sigma |S_k| \le |y|\le Q_1 l_k^{1/2}$,
\begin{align*}
&\left(\Delta + \bar b_i(y)\partial_i+d_{ij}\partial_{ij}-\bar c (y)\right) f_{\lda}^{(1)}(y) +(N-1) E_\lda^{(1)}(|y|)\\
\le &\Delta f_{\lda}^{(1)}(y)+
|\bar b_i(y)\partial_i f_{\lda}^{(1)}(y)+d_{ij}(y)\partial_{ij}f_{\lda}^{(1)}(y) -\bar c(y) f_{\lda}^{(1)}(y)|+(N-1) E_\lda^{(1)}(|y|)\\
=&-E_\lda^{(1)}(|y|)+ o(1)(E_\lda^{(1)}(|y| )+E_\lda^{(2)}(y))\\
\le &E_\lda^{(2)}(y).
\end{align*}
This completes our proof of Lemma \ref{lem:h-est-1}.
\end{proof}

Next we construct auxiliary functions to handle the $E_{\lda}^{(2)}(y)$ term.
To exploit the conformal normal property of $\bar g_k (y)$ inside $B_{\sigma |S_k|/2}$ so that
we can use radial auxiliary functions as much as possible,
we modify $\bar g_k (y)$ outside of $B_{\sigma |S_k|/2}$ into  $\hat g_k(y)$
as a smooth Riemann metric defined on $\R^n$ satisfying
\[
\hat g_k(y)=\bar g_k(y)\quad \mbox{in }B_{\sigma |S_k|/2}, \quad (\hat g_k)_{ij}(y)= \delta_{ij} \quad \mbox{in }\R^n\setminus B_{\sigma |S_k|},
\]
\[
\det \hat g_k(y)=1 \quad \mbox{in }\R^n,
\]
and
\[
|\nabla^ m ((\hat g_k)_{ij}(y)-\delta_{ij})| \le C l_k^{-2} |x_k|^{\tau-2} |y|^{2-m} \quad
\mbox{for } y\in B_{\sigma |S_k|}
\mbox{ and }m=0,1,2,
\]
where this last estimate is due to \eqref{eq:varkappa} of Lemma~\ref{lem:conf-normal}  applied to
$g_{k}(x)$ for $|x|<\sigma$ with $\bar z=x_{k}$, and  $\bar g_{k}(y)=g_{k}(l_{k}^{-1} y)$.
In particular, $|(\hat g_k)_{ij}(y)-\delta_{ij}|\le C|x_k|^{\tau} \to 0$ uniformly over $\R^n$ as $k\to \infty$.

We will also require that $y$ be a geodesic coordinate system for $\hat g_k$ on $\R^n$.
This can be
done by first expressing $\bar g_k(y) =\exp ((\bar h_k)_{ij}(y))$ for $y\in B_{\sigma |S_k|}$,
 where $(\bar h_k)_{ij}(y)= (h_k)_{ij}(l_{k}^{-1}y)$ satisfies
 (a) $\sum_{j=1}^n (\bar h_k)_{ij}(y)y_j=0$,
and (b) $\text{trace} ((\bar h_k)_{ij}(y))=0$ in $B_{\sigma |S_k|}$. Then it is trivial to
modify $(\bar h_k)_{ij}(y)$ outside of $B_{\sigma |S_k|/2}$ and
extend $(\bar h_k)_{ij}(y)$ to
$y\in R^n$ such that $(\bar h_k)_{ij}(y)=0$ for $y\in \R^n\setminus B_{\sigma |S_k|}$ while maintaining the
two conditions (a) and (b), which guarantee that $y$ is a
geodesic coordinate system for $\hat g_k$ and its determinant $\equiv 1$ on $\R^n$ .

Let  $\eta_k^{(1)}$ be the smooth solution of
\[
-\Delta_{\hat g} \eta_{k,\lda}^{(1)}(y)=|S_k|^{2-n} |y-S_k|^{-\frac{6-n}{2}} \chi_k(y)  \quad
\mbox{in }\R^n\setminus B_\lda,
\]
with the boundary condition
\[
\eta_{k,\lda}^{(1)}=0 \quad  \mbox{on }\pa B_\lda, \quad  \mbox{and}\quad  \lim_{|y|\to \infty} \eta_{k,\lda}^{(1)}(y)=0.
\]
By the upper bound estimate on the Green's function $\hat G_{k}(y, z)$ of $\Delta_{\hat g}$
due to Littman, Stampacchia and Weinberger \cite{LSW} (see also Grüter and Widman
\cite{GW}),
\[
\hat G_{k}(y, z)\le C |y-z|^{2-n} \quad \text{for some $C>0$ independent of $k$, and all
$y, z$ with $\lda < |y|, |z|$,}
\]
  we have
\begin{align*}
0\le \eta_{k,\lda}^{(1)}(y)\le &C |S_k|^{2-n} \int_{B_{|S_k|/2} (S_k)} |y-z|^{2-n}  |z-S_k|^{-\frac{6-n}{2}}  \,\ud z\\& \le
C_2\begin{cases}
|S_k|^{-\frac{n-2}{2}} &\quad \mbox{for }|y|<2|S_k|,\\[2mm]
|y|^{2-n} |S_k|^{\frac{n-2}{2}}&\quad \mbox{for }|y|\ge 2|S_k|.
\end{cases}
\end{align*}
Furthermore, $0\le \pa_r \eta_{k,\lda}^{(1)}(y)\le C_2 |S_k|^{-\frac{n-2}{2}}$ on  $\pa B_\lda$,
where $C_2>0$ is independent of $\lda$ and $k$ (if $k$ is large),  and 
\[
\eta_{k,\lda}^{(1)}(y) \le 2C_2 |S_k|^{-\frac{n-2}{2}} (|y|-\lda) \quad \mbox{for } \lda\le |y|\le \lda+2.
\] Let
\[
\eta_{k,\lda}^{(2)}(|y|)= Q\cdot C_2|S_k|^{-\frac{n-2}{2}}  (\lda^{n-2} |y|^{2-n}-1),
\]
where $Q>0$ is depending only on $n$ such that $\eta_{k,\lda}:= \eta_{k,\lda}^{(1)}+\eta_{k,\lda}^{(2)}\le 0$ in $\R^n\setminus B_\lda$.
Since $y$ is a conformal normal coordinate for $ \hat g_k$ on
$\R^n$, $\Delta_{\hat g_k} \eta_{k,\lda}^{(2)}= \Delta \eta_{k,\lda}^{(2)}=0$.
Hence,
\[
-\Delta_{\hat g_k} \eta_{k,\lda}(y)= |S_k|^{2-n} |y-S_k|^{-\frac{6-n}{2}} \chi_k(y) \ge 0 \quad \mbox{in }\R^n\setminus B_\lda,
\]
\[
\eta_{k,\lda}=0, \quad \pa_r \eta_{k,\lda} <0 \quad \mbox{on }\pa B_\lda.
\]
Moreover,
\be \label{eq:diff-eta-k}
|\nabla^m  \eta_{k,\lda}(y)| \le C \begin{cases}
|S_k|^{-\frac{n-2}{2}}
|y-S_k|^{-m}, & \quad \frac{1}{2}\sigma |S_k| \le |y|\le 2 |S_k|, \\[2mm]
|S_k|^{-\frac{n-2}{2} }|y|^{2-n-m},  &\quad |y|\ge 2|S_k|,
\end{cases}
\ee
for $m=1,2$.
Let \[
f_{\lda}^{(2)}(y)= (2C_1+1) l_k^{-\tau}\eta_{k,\lda}(y).
\]
We have dropped the subscript $k$ of $f_{\lda}^{(2)}(y)$ here.

\begin{lem} \label{lem:h-est-2}  Let $f_{\lda}^{(2)}(y)$ be defined above. Then we have
\be \label{eq:lem38a}
f_{\lda}^{(2)}(y) \le 0 \quad \mbox{in  }\R^n\setminus B_\lda, \quad f_{\lda}^{(2)}(y) =0 \quad \mbox{on }\pa B_\lda,
\ee
and for  $\lambda<|y|<Q_1l_k^{1/2}$,
\be \label{eq:lem38b}
\left[ \Delta +\bar b_i(y)\partial_i+\bar d_{ij}(y)\partial_{ij}-\bar c(y)\right] f_{\lda}^{(2)}(y)+2C_1 E_\lda^{(2)}(y)  \le o(1)E_\lda^{(1)}(y) ,
\ee
\be  \label{eq:lem38c}
|f_{\lda}^{(2)}(y)|=o(1) M_k^{-1}\quad\text{uniformly over $\lambda<|y|<Q_1l_k^{1/2}$.}
\ee
\end{lem}

\begin{proof} By the definition of $f_{\lda}^{(2)}$ and its construction, \eqref{eq:lem38a} and \eqref{eq:lem38c} hold.
It remains to show \eqref{eq:lem38b}.

 If $|y|\le \frac{1}{2} \sigma |S_k|$, we have $\hat g_k=\bar g_k$, and thus
 \begin{align*}
&|\left[ \Delta +\bar b_i(y)\partial_i+\bar d_{ij}(y)\partial_{ij}-\bar c(y)\right]f_{\lda}^{(2)}(y)|\\[2mm]&
= |\bar c(y) f_{\lda}^{(2)}(y) |\\&
\le C l_k^{-4}|x_k|^{\tau-4}|y|^2  \cdot l_k^{-\tau}|S_k|^{\frac{2-n}{2}}.
\end{align*}
It follows that
\[
\frac{|\left[ \Delta +\bar b_i(y)\partial_i+\bar d_{ij}(y)\partial_{ij}-\bar c(y)\right]f_{\lda}^{(2)}(y)| }{E_{\lda}^{(1)}(y)}
\le C |y|^{n-2} l_k^{-\tau}|S_k|^{\frac{2-n}{2}}  \to 0
\]
uniformly over $|y|\le \frac{1}{2} \sigma |S_k|$ as $k\to \infty$.

Let us note that
\[
E_{\lda}^{(1)}(|y|) \ge l_k^{-4} |x_k|^{\tau-4}|y|^{4-n}  \ge \frac{1}{C} |x_k|^{\tau}|S_k|^{-(n-2) -2}\ge
|x_k|^{-\tau} M_k^{-2} |S_k|^{-2}
\]
for $\frac{1}{2}\sigma |S_k| \le |y|\le \sigma |S_k| $, and  recall  \eqref{eq:E-lda-low}
\[
E_{\lambda}^{(1)}(|y|)+E_{\lda}^{(2)}(y)\ge \frac{1}{C}l_k^{-\tau} |y-S_k|^{\tau-2}|l_k^{1/2}|^{2-n}
\]
for  $ \sigma |S_k| \le |y|\le Q_1 l_k^{1/2}$. Making use of \eqref{eq:diff-eta-k}, the properties of $\hat g_k$, and the coefficients estimates \eqref{bik}, \eqref{dijk} and \eqref{ck}, we obtain
\[
|\left( \Delta-\Delta_{\hat g_k}\right) f_{\lda}^{(2)}(y)| \le C M_k^{-2} |S_k|^{-2} =o(1)E_\lda^{(1)}(|y|)\quad \mbox{if }\frac12\sigma |S_k|\le  |y|\le \sigma |S_k|,
\]
\[
|(\Delta-\Delta_{\hat g_k}) f_{\lda}^{(2)}(y)| =0 \quad \mbox{if }|y|> \sigma |S_k|,
\]
and
\begin{align*}
&|(\bar b_i(y)\partial_i+\bar d_{ij}(y)\partial_{ij}-\bar c(y))f_{\lda}^{(2)}(y)|\\&
\le C \begin{cases}
 M_k^{-2} |S_k|^{-2} ,& \quad \frac{1}{2}\sigma |S_k| \le |y|\le \sigma |S_k|\\
l_k^{-\tau}|y-S_k|^{\tau-2} \cdot l_k^{-\tau}|S_k|^{\tau+2-n},& \quad  \sigma |S_k| \le |y|\le Q_1 l_k^{1/2}
\end{cases}\\&=o(1)\left(E_{\lambda}^{(1)}(|y|)+E_{\lda}^{(2)}(y)\right).
\end{align*}
Hence,
\[
\left[ \Delta +\bar b_i(y)\partial_i+\bar d_{ij}(y)\partial_{ij}-\bar c(y)\right]f_{\lda}^{(2)}(y) = -2(C_1+1) E_\lda^{(2)}(y)+o(1)\left(E_{\lambda}^{(1)}(|y|)+E_{\lda}^{(2)}(y)\right),
\]
which implies \eqref{eq:lem38b}. This completes our  proof of Lemma \ref{lem:h-est-2}.

\end{proof}








\begin{prop} \label{prop:final} Let $f_{\lda}= f_{\lda}^{(1)}+f_{\lda}^{(2)}$.
By taking $N>2C_1+2$ in \eqref{eq:barrier-ode},   we have
\[
\left[ \Delta +\bar b_i(y)\partial_i +\bar d_{ij}(y)\partial_{ij}-\bar c(y)\right] f_{\lda}(y)+\xi_\lda(y) f_{\lda}(y)
\le -|E_{\lambda}(y)| \quad \mbox{in }\Sigma_{\lda}^k,
\]
\be \label{fsign}
f_{\lda}(y)=0\quad\text{on $\partial B_{\lda}$,}\quad f_{\lda}(y)<0\quad\text{in $\Sigma_{\lda}^k$,}
\ee
\be \label{fcontrol}
|f_{\lda}(y)|+|\nabla f_{\lda}(y)|=o(1) M_k^{-1} \quad\text{uniformly in $\Sigma_{\lda}^k$ as $k\to \infty$.}
\ee
\end{prop}
We remark that when $y\in \Sigma_{\lda}^k$, $|y|\le Q_1 l_k^{1/2}$, so $|y|^{2-n}\ge Q_1^{2-n} l_k^{\frac{2-n}{2}}$,
and the estimate $|f_{\lda}(y)|=o(1) M_k^{-1}$ certainly implies \eqref{fest}, namely,
$|f_{\lda}(y)|+|\nabla f_{\lda}(y)|=o(1) |y|^{2-n}$ on $\Sigma_{\lda}^k$.

\begin{proof}  By Lemma \ref{lem:h-est-1} and Lemma \ref{lem:h-est-2}, we have
\begin{align*}
&\left[ \Delta +\bar b_i(y)\partial_i +\bar d_{ij}(y)\partial_{ij}-\bar c(y)\right]f_{\lda}(y)+\xi_\lda(y) f_{\lda}(y) \\&
\le -(N- 1) E_\lda^{(1)}(y) + E_\lda^{(2)}(y) -(2C_1+1) E_\lda^{(2)}(y) +o(1)E_\lda^{(1)}(y)  \\&
=-C_1 (E_\lda^{(1)}(y) +E_\lda^{(2)}(y))\le -|E_\lda(y)|,
\end{align*}
where we have used $\xi_\lda(y)\ge 0$.

\eqref{fsign} and \eqref{fcontrol} follow from our construction of
$f_{\lda}^{(1)}$ and $f_{\lda}^{(2)}$ in
Lemmas  \ref{lem:h-est-1} and \ref{lem:h-est-2} respectively. This completes our proof.
\end{proof}

\subsection{Case of $ n\ge  7$.}

The way we handled the term $\bar c(y)v_k^{\lambda}(y)- (\frac{\lambda}{|y|})^{n+2}
\bar c(y^{\lambda})v_k(y^{\lambda})$ for $|y|\le \frac{\sigma}{2}|S_k|$ in constructing
$f_{\lda}^{(1)}(r)$ when $3\le n \le 6$ is
no longer adequate for the  $n\ge 7$ cases.
We will need to modify the construction of $f_{\lda}^{(1)}(r)$ here, which uses more delicate estimates.
The key is to identify leading order terms as given in \eqref{eq:724-h} and construct the auxiliary function
$f_\lda^{(3)}$ with respect to these leading order terms as in \eqref{f3sol}.
In this process we need to obtain refined estimates for $u_k(z)$ near $z=0$,
or equivalently for $v_k(y)-U(y)$ in terms
of $|x_k|$ and $|S_k|^{-1}$, as given by Lemma \ref{lem:sigma-k}.

Since  $\tau=\frac{n-2}{2}>2$ if $n\ge 7$, we can expand the conformal  Laplacian as
\[
L_{\bar g_k}= \Delta +\bar b_i(y)\partial_i +\bar d_{ij}(y)\partial_{ij}-\bar c(y)
\]
with
\begin{align}\label{bik-2}
\bar b_i(y)=&l_k^{-1}b_i(l_k^{-1}y)\nonumber \\[2mm]
=& \begin{cases}
O(l_k^{-2}|x_k|^{\tau-2}|y|),&\quad |y|<\sigma  |S_k|,\\[2mm]
O(M_k^{-1}|y|^{\tau-1}),&\quad |y|>\sigma  |S_k|,
\end{cases}
\end{align}
\begin{align}\label{dijk-2}
\bar d_{ij}(y)&=d_{ij}(l_k^{-1}y)\nonumber \\[2mm]& =
\begin{cases}
O(l_k^{-2}|x_k|^{\tau-2}|y|^2),\quad |y|<\sigma  |S_k|,\\[2mm]
O(M_k^{-1}|y|^{\tau}),\quad |y|>\sigma  |S_k|,
\end{cases}
\end{align}
and
\begin{align}\label{ck-2}
\bar c(y)&=c(n)R_{\bar g_k}(y)=c(n)l_k^{-2}R_{g_k}(l_k^{-1}y)\nonumber \\&= \begin{cases}
O(l_k^{-4}|x_k|^{\tau-4}|y|^2),\quad |y|<\sigma  |S_k|,\\[2mm]
O(M_k^{-1}|y|^{\tau-2}),\quad |y|>\sigma |S_k|.
\end{cases}
\end{align}
Recall that
\begin{eqnarray*}
E_{\lambda}(y)&=& \left(\bar c(y)v_k^{\lambda}(y)- (\frac{\lambda}{|y|})^{n+2}
\bar c(y^{\lambda})v_k(y^{\lambda})\right)
-\Big(\bar b_i(y)\partial_i v_k^{\lambda}(y)+\bar d_{ij}(y)\partial_{ij}v^{\lambda}_k(y)\Big)\\
&& +(\frac{\lambda}{|y|})^{n+2}\left(\bar b_i(y^{\lambda})
\partial_i v_k(y^{\lambda})+
\bar d_{ij}(y^{\lambda})\partial_{ij}v_k(y^{\lambda})\right).
\end{eqnarray*}

We rescale $g_k$ and $u_k$ so that the local maximum of $u_k$ at $z=0$ and its singular point at
$z=-x_k$ are of unit distance.
Let $(\varpi_k)_{ij}(z)= (g_k)_{ij}(|x_k| z)$ for $|z|\le 1$.
By \eqref{eq:coord-u}, we have
\[
\det \varpi_k=1 \quad \mbox{in }B_{\sigma}.
\]
Moreover, $\|\varpi _k\|_{C^{n+2}(B_{\sigma})}\le C$. We shall apply the results in Section \ref{sec:blow-up} to $(\varpi_k)_{ij}(z)$.
Write  $(\varpi_k)_{ij}(z)=e^{h_{ij}(z)}$, where we dropped the subscript $k$ of $h_{ij}$.
Define $H_{ij}, \tilde H_{ij}$,  $\tilde Z_\va$ and $Z_\va$  etc.\ as there.
Since \eqref{eq:varkappa} holds for $g_{k}(x)$ with $\bar z=x_{k}$, it follows that,
  for $z\in B_\sigma$,
\be\label{eq:724-a}
|\nabla ^m h_{ij}(z)|\le C |x_k|^{\tau}, \quad m=0,\dots, n+2,
\ee
for some $C$ independent of $k$.
 Furthermore, it follows from \eqref{eq:slc-expn-1} and \eqref{eq:724-a} that
\be  \label{eq:724-b}
\begin{split}
|R_{\varpi_k}(z)-\pa_i\pa_j H_{ij}(z)| &\le C \sum_{|\al|=2}^{d} \sum_{i, j}
 |h_{ij \al}|^2 |z|^{2|\al|-2} +C\|h_{ij}\|_{C^{n-2}(B_{\sigma})} |z|^{n-5} \\&
\le C( |x_k|^{2\tau} + |x_k|^{\tau}|z|^{n-5}) ,
\end{split}
\ee
 where $  H_{ij}=\sum_{l=4}^{n-4} \pa_i\pa_j  H_{ij}^{(l)} $ if $n\ge 8$ while $ H_{ij} =0$ otherwise.
Define
\[
\psi_k(z)=|x_k|^{\frac{n-2}{2}} u_k(|x_k| z), \quad  |z|\le \frac1{2|x_k|}.
\] Then
\[
-L_{\varpi_k}\psi_k= n(n-2) \psi_k^{\frac{n+2}{n-2}} \quad \mbox{in }B_{\frac1{2|x_k|}}\setminus \{-\frac{x_k}{|x_k|}\},
\]
 $0$ is a local maximum point of $\psi_k$ and $\psi_k(0)=|S_k|^{\frac{n-2}{2}}\to \infty$ as $k\to \infty$.

 If $n\ge 25$, we shall verify \eqref{eq:weyl-vanishing-2} in the current setting:
 \be \label{eq:weyl-vanishing-3}
 \sum_{l=1}^{d}  |\nabla ^l \varpi_k(\bar z)|^2 \va_{\bar z, k}^{2l} |\ln \va_{\bar z, k} |^{\theta_{l}}
 = o(1)\va_{\bar z, k}^{n-2}
 \ee
 for any $\bar z\in B_{\sigma}$ with $\psi_k(\bar z)\ge 1$, where
 $\va_{\bar z, k} = \psi_k(\bar z)^{-\frac{2}{n-2}}$.
 By \eqref{eq:724-a} and \eqref{eq:25-wk-bd-1}, we have,  for any point $\bar z\in B_{\sigma}$,
\[
|\nabla ^l  \varpi_k(\bar z)|^2  \le C |x_k|^{2\tau} \le C u_{k}(|x_k|\bar z)^{-1} , \quad l=1,2\dots, n+2,
\]
here, we have used $|x_{k}|\le C|\phi_{x_{k}}(|x_{k}|\bar z)|$ for
$\bar z \in B_{\sigma}$. Furthermore, using
\[
\va_{\bar z, k}=|x_{k}|^{-1} u_{k}(|x_k|\bar z)^{-\frac{2}{n-2}},
\]
we get
\[
u_{k}(|x_k|\bar z)^{-1}= \va_{\bar z, k}^{\frac{n-2}{2}}|x_{k}|^{\frac{n-2}{2}}\le
C \va_{\bar z, k}^{\frac{n-2}{2}} u_{k}(|x_k|\bar z)^{-1/2},
\]
from which we get
\[
|x_k|^{n-2}\le C u_{k}(|x_k|\bar z)^{-1}\le C \va_{\bar z, k}^{n-2}.
\]
Then, for $1\le l \le d$,
\[
|\nabla ^l \varpi_{k}(\bar z)|^2  \va_{\bar z, k}^{2l} \le C |x_{k}|^{n-2}  \va_{\bar z, k}^{2l}
\le C  \va_{\bar z, k}^{2l+n-2},
\]
and for $l=d$,
\[
|\nabla ^d \varpi_{k}(\bar z)|^2  \va_{\bar z, k}^{2d} |\ln \va_{\bar z, k} |
\le C \va_{\bar z, k}^{2(n-2)} |\ln \va_{\bar z, k} |=o(1) \va_{\bar z, k}^{n-2}.
\]
It is now clear that \eqref{eq:weyl-vanishing-3} holds.

 We are now ready to deal with the case of $ n \ge 7$.
 It follows from Proposition \ref{prop:KMS-8.2}
 that  $0$ is an isolated simple blow up point of $\psi_k$ with some $\rho>0$ independent of $k$. We may take $\rho=\sigma/2$ without loss of generality.  Note that
 \[
 v_k(y)= |S_k|^{-\frac{n-2}{2}} \psi_k(|S_k|^{-1} y)
 \]
 and $\bar g_{k}(y)= \varpi_k(|S_k|^{-1} y)$.


\begin{lem}\label{lem:sigma-k} Let $\va_k= |S_k|^{-1}$.  We have
\[
|\nabla^m(v_k-U-\tilde Z_{\va_k})(y)|\le C\va_k^{n-3}(1+|y|)^{-1-m} , \quad |y|\le \frac{\sigma}{2} |S_k|,
\]
where $C>0$, $m=0,1,2$, and $\tilde Z_{\va_k}$ solves \eqref{eq:corr-1} with  the bound
\[
|\nabla ^m \tilde Z_{\va_k}(y) |\le C\min\{ |x_k|^{\tau} \va_k ^{4},
\va_k^{\frac{n-2}{2}} \}(1+|y|)^{6-n-m}.
\]
\end{lem}

\begin{proof}
It follows by applying Proposition \ref{prop:KMS-8.2} and Proposition~\ref{prop:isol-to-isol-sim}
to the solution $\psi_k(x)$ with respect to the metric $\varpi_k$---note that $v_k(y)$ is the normalization
of $\psi_k(x)$ by $|S_k|$.
\end{proof}


Let
\begin{align*}
V_\lda(r)= \begin{cases}
n(n-2)\frac{U(r)^{\frac{n+2}{n-2}}-U^\lda(r)^{\frac{n+2}{n-2}} }{U(r)-U^\lda(r)}& \quad \mbox{if }\lda\neq 1, \\[2mm]
n(n+2) U(r)^{\frac{4}{n-2}} &\quad \mbox{if }\lda =1.
\end{cases}
\end{align*}
Define
\[
\mathcal{O}_\lda= \{y\in B_{Q_1 l_k^{1/2}} \setminus B_\lda: v_k(y)\le 2v_k^{\lda}(y) \},
\]
where $Q_1>0$ is the constant defined in \eqref{eq:Q-1}. It is easy to see that $\mathcal{O}_\lda \subset \subset  B_{Q_1l_k^{1/2}}$. $V_\lda(y)$ provides a good approximation for $\xi_\lda (y)$, as given below.
\begin{lem} \label{lem:efft-estm}  Let $\xi_\lda $ be the function defined in \eqref{eq:xi-def}. Then we have
\[
|\xi_\lda (y)-V_\lda(y)|\le C|S_k|^{-\frac{n-2}{2} } |y|^{n-6} \quad \mbox{for }\lda\le  |y|\le \frac{\sigma }{2} |S_k|
\]
and
\[
|\xi_\lda (y)-V_\lda(y)| \le C |y|^{-4} \quad \mbox{for }y\in \mathcal{O}_\lda.
\]
\end{lem}
\begin{proof} By Lemma \ref{lem:sigma-k}, we have
\[
a_k(y):=v_k(y)-U(y)=O(\va_k^{\frac{n-2}{2}})\quad\mbox{and}\quad
b_k(y):= v_k^\lda(y)-U^\lda(y) =O(\va_k^{\frac{n-2}{2}})
\]
if $|y|\le \frac{\sigma }{2} |S_k|$.
By  direct calculus, we have
\begin{align*}
&\frac{(U(|y|)+a_k(y))^{\frac{n+2}{n-2}}-(U^\lda(|y|) +b_k(y))^{\frac{n+2}{n-2}} }{(U(|y|)+a_k(y))-(U^\lda(|y|)+b_k(y))}\\&
=\frac{n+2}{n-2} \int_0^1 \Big(t U(|y|)+ (1-t)U^\lda(|y|) \Big)^{\frac{4}{n-2}} \,\ud t + O(1)(|a_k(y)|+|b_k(y)|) |y|^{n-6}\\&
=\frac{1}{n(n-2)} V_\lda +O(1) \va_k^{\frac{n-2}{2}} |y|^{n-6}.
\end{align*}
This proves the first inequality of Lemma \ref{lem:efft-estm}. The second inequality is obvious.

\end{proof}

We now work out the leading order terms as described at the beginning of this subsection. First,
by \eqref{eq:724-b} and Proposition \ref{prop:x-k-estimates}, we have
\be \label{eq:724-c}
\begin{split}
\left|R_{\bar g_k}(y)-|S_k|^{-2} \pa_i \pa_j H_{ij}(|S_k|^{-1}y)\right|&\le C  (|x_k|^{2\tau} |S_k|^{-2} +|x_k|^\tau |S_k|^{3-n} |y|^{n-5} )\\&
\le C  M_k^{-1} ( |S_k|^{-2} +|S_k|^{\frac{4-n}{2}} |y|^{n-5} ) .
\end{split}
\ee

It follows that for  $\lda\le |y|\le \frac{\sigma }{2} |S_k|$,  $\lda\in [1/2,2]$,
\begin{align}
&R_{\bar g_k}(y) U^\lda(y)- (\frac{\lda}{|y|})^{n+2} R_{\bar g_k}(y^\lda) U(y^\lda) \nonumber \\
&= |S_k|^{-2}( \pa_i\pa_j  H_{ij} (|S_k|^{-1} y) U^\lda(y)-(\frac{\lda}{|y|})^{n+2}    \pa_i\pa_j  H_{ij}(|S_k|^{-1} y^{\lda}) U(y^\lda))\nonumber  \\& \quad +O(1)M_k^{-1} ( |S_k|^{-2}|y|^{2-n} +|S_k|^{\frac{4-n}{2}} |y|^{-3} ).\nonumber \\&
= \lda ^{n-2} \sum_{l=4}^{n-4} |S_k|^{-l} |y|^{l-n} (1-(\frac{\lda}{|y|})^{2l}) \pa_i\pa_j  H_{ij}^{(l)}(\frac{y}{|y|}) U(y^\lda)+O(1)M_k^{-1} |S_k|^{-\frac{3}{2}}|y|^{-3}  .
\label{eq:724-d}
\end{align}
By \eqref{ck-2},  for $ |y|\ge \frac{\sigma }{2} |S_k|$
\begin{align*}
 &R_{\bar g_k}(y) U^\lda(y)- (\frac{\lda}{|y|})^{n+2} R_{\bar g_k}(y^\lda) U(y^\lda)\\&
 \le C(M_k^{-1}|y|^{\tau-n} +  l_k^{-4}|x_k|^{\tau-4} |y|^{-4-n} ) \le C M_k^{-1}|y|^{-\frac{n+2}{2}} .
\end{align*}
Let
\[
T_k(y) =\begin{cases}
 \lda ^{n-2} \sum_{l=4}^{n-4} |S_k|^{-l} |y|^{l-n} (1-(\frac{\lda}{|y|})^{2l}) \pa_i\pa_j  H_{ij}^{(l)}(\frac{y}{|y|}) U(y^\lda),& \quad \lda\le  |y|<\frac{\sigma}{2} |S_k|,\\[2mm]
 \lda ^{n-2} \sum_{l=4}^{n-4} (\frac{\sigma}{2})^l |y|^{-n} (1-(\frac{\lda}{|y|})^{2l}) \pa_i\pa_j  H_{ij}^{(l)}(\frac{y}{|y|}) U(y^\lda),& \quad |y|\ge \frac{\sigma}{2} |S_k|.
\end{cases}
\]
Hence,
\be \label{eq:724-h}
\begin{split}
&R_{\bar g_k}(y) U^\lda(y)- (\frac{\lda}{|y|})^{n+2} R_{\bar g_k}(y^\lda) U(y^\lda)\\&
 = T_k(y)  + O(1)\begin{cases}
M_k^{-1} |S_k|^{-\frac{3}{2}}|y|^{-3}, & \quad \lda\le  |y|<\frac{\sigma}{2} |S_k|,\\[2mm]
M_k^{-1}|y|^{-\frac{n+2}{2}} , & \quad   |y|\ge \frac{\sigma}{2} |S_k|.
 \end{cases}
\end{split}
\ee

Since
\[
\int_{|y|=1}   \pa_i\pa_j  H_{ij}^{(l)}(y) =\int_{|y|=1}  y_a \pa_i\pa_j  H_{ij}^{(l)}(y)=0, \quad a=1,\dots, n,
\]
we have
\[
 \pa_i\pa_j H_{ij}^{(l)}(\frac{y}{|y|})= \sum_{ s=2}^{ l-2} Y_{l,s}(\frac{y}{|y|}),
\]
where $Y_{l,s} $ are some spherical harmonics of degree $s$ on $\mathbb{S}^{n-1}$,
orthogonal to each other in $L^{2}(\mathbb S^{n-1})$.
Let
\[
T_k^{(l)}(r) =\begin{cases}
 \lda ^{n-2} |S_k|^{-l} r^{l-n} (1-(\frac{\lda}{r})^{2l})  U(\frac{\lda^2}{r}),& \quad \lda\le  r<\frac{\sigma}{2} |S_k|,\\[2mm]
 \lda ^{n-2}  (\frac{\sigma}{2})^l |r|^{-n} (1-(\frac{\lda}{r})^{2l}) U(\frac{\lda^2}{r}) ,& \quad r\ge \frac{\sigma}{2} |S_k|.
\end{cases}
\]
By Proposition 6.1 of \cite{LZ2}, there exists a small constant  $\delta_1>0$ depending only on $n$,  such that, for every  $\lambda\in [1-\delta_1, 1+\delta_1]$, the boundary value problem
$$f_{l,s}''(r)+\frac{n-1}rf_{l,s}'(r)+(V_{\lambda}(r)-\frac{s(s+n-2)}{r^2})f_{l,s} =-c(n)T_k^{(l)}(r) ,\quad \lambda <r<Q_1 l_k^{1/2} $$
with
\[
f_{l,s} (\lambda)=f_{l,s} (Q_1 l_k^{1/2})=0,
\]
has a unique solution. Moreover,
\[
|f_{l,s}(r)|\le \begin{cases}
C|S_k|^{-4}(1+r)^{6-n},&\quad \lambda \le r\le |S_k|, \\
Cr^{2-n},& \quad |S_k|\le r\le Q_1 l_k^{1/2}.
\end{cases}
\]
We remark that the restriction on $\lda \in [1-\delta_1, 1+\delta_1]$ is
to ensure that $V_\lda(|y|) $ is sufficiently close to $V_1(r)=n(n+2)U(r)^{\frac{4}{n-2}}$.
Let
\[
f_{\lambda}^{(3)}(y)=\sum_{l=4}^{n-4} \sum_{ s=2}^{l-2} f_{l,s} (|y|) Y_{l,s} (\frac{y}{|y|}).
\]
Then
\be \label{f3sol}
\begin{cases}
\Delta f_\lda^{(3)}(y) +V_\lda(|y|) f_\lda^{(3)}(y) = -c(n) T_k(y)  \quad &\mbox{in }B_{Q_1 l_k^{1/2}} \setminus B_\lda,  \\
f_\lda^{(3)}(y) =0  \quad &\mbox{on }\pa (B_{Q_1 l_k^{1/2} } \setminus B_\lda).\\
\end{cases}
\ee
Using \eqref{eq:724-a}, we have
\[
 \max_{|e|=1}|\pa_i\pa_j  H_{ij}^{(l)}(\frac{y}{|y|})|\le C |x_k|^\tau \quad \text{and} \quad
 | Y_{l,s}(\frac{y}{|y|})| \le C |x_k|^\tau,
\]
which leads to the following crucial estimate for $f_\lda^{(3)}(y)$:
\be \label{eq:724-e}
|\nabla^m f_\lda^{(3)}(y)| \le
\begin{cases}
C |S_k|^{-4} |x_k|^{\tau} (1+|y|)^{6-n-m} ,& \quad |y|<\frac{\sigma}{2} |S_k|,\\[2mm]
C |x_k|^{\tau} (1+|y|)^{2-n-m} ,& \quad |y|\ge \frac{\sigma}{2} |S_k|,
\end{cases}
\ee
for $m=0, 1, \cdots, n+2$.  Next we have
\begin{align*}
&|(\bar b_i(y)\partial_i +\bar d_{ij}(y)\partial_{ij}-\bar c(y))f^{(3)}_{\lambda} (y)|
\le C\begin{cases}
 |S_k|^{-6} |x_k|^{2\tau} |y|^{6-n} ,\quad |y|<\frac{\sigma}{2}  |S_k|,\\[2mm]
  M_k^{-1} |x_k|^{\tau}|y|^{\tau -n} ,\quad |y|>\frac{\sigma}{2}   |S_k|.
\end{cases}
\end{align*}
By Lemma \ref{lem:efft-estm},
\begin{align*}
 |(V_\lda(y)-\xi_\lda(y))f^{(3)}_{\lambda}(y)| \le C \begin{cases}
 |S_k|^{-\frac{n+6}{2}} |x_k|^{\tau}  ,\quad |y|<\frac{\sigma}{2}    |S_k|,\\[2mm]
   |x_k|^{\tau} |y|^{-2-n}  ,\quad |y|>\frac{\sigma}{2}   |S_k|, ~y\in  \mathcal{O}_\lda.
\end{cases}
\end{align*}
Since $|x_k|\le C l_k^{-1/2}$ and $|S_k|= l_k |x_k|$, it follows that
\be  \label{eq:724-i}
\begin{split}
&\left[ \Delta +\bar b_i(y)\partial_i +\bar d_{ij}(y)\partial_{ij}-\bar c(y)\right]f^{(3)}_{\lambda}(y)
+\xi_\lda(y) f^{(3)}_{\lambda}(y) \\&
= (\Delta +V_\lda(y)) f^{(3)}_{\lambda}(y) +|(\bar b_i(y)\partial_i +\bar d_{ij}(y)\partial_{ij}-\bar c(y))f^{(3)}_{\lambda}(y) |
+ |V_\lda(y)-\xi_\lda(y)| |f^{(3)}_{\lambda}(y)| \\&
= (\Delta +V_\lda(y)) f^{(3)}_{\lambda}(y) +O(1) \begin{cases}
M_k^{-1} |S_k|^{-4} ,\quad |y|<\frac{\sigma}{2}    |S_k|,\\[2mm]
 M_k^{-1}  |y|^{-\frac{n+6}{2}}  ,\quad |y|>\frac{\sigma}{2}   |S_k|, ~y\in  \mathcal{O}_\lda.
\end{cases}
\end{split}
\ee

Let
\[
\hat E_\lda(y)= \left[\Delta +\bar b_i(y)\partial_i +\bar d_{ij}(y)\partial_{ij}-\bar c(y)+\xi_\lda(y)\right] \big(w_\lda+ f^{(3)}_{\lambda} \big)(y), \quad |y|\ge \lda,
\]
where $w_\lda(y)=v_k(y)-v_k^{\lda}(y)$. Recall that
$\left[ \Delta +\bar b_i(y)\partial_i +\bar d_{ij}(y)\partial_{ij}-\bar c(y)+\xi_\lda(y)\right] w_\lda(y) =E_\lda(y)$.
We have
\be \label{prop:error-est-2}
|\hat E_\lda (y)|\le C _2 \begin{cases}
M_k^{-1} |S_k|^{-1}|y|^{-3} ,\quad \lda< |y|<\frac{\sigma}{2}    |S_k|,\\[2mm]
 M_k^{-1}  |y|^{-\frac{n+2}{2}}  ,\quad |y|>\frac{\sigma}{2}   |S_k|, ~y\in  \mathcal{O}_\lda.
\end{cases}
\ee

\begin{proof}[Proof of \eqref{prop:error-est-2}]
 By Lemma \ref{lem:sigma-k}, we have
\be \label{eq:724-g}
|\nabla^m(v_k-U)(y)|\le C|S_k|^{\frac{2-n}{2}} \quad \mbox{for }|y|<\frac{\sigma}{2} |S_k|.
\ee
By \eqref{eq:724-h} and \eqref{ck-2}, we have
\begin{align*}
&\bar c(y)v_k^{\lambda}(y)- (\frac{\lambda}{|y|})^{n+2}
\bar c(y^{\lambda})v_k(y^{\lambda}) \\
&= \bar c(y)U^{\lambda}(y)- (\frac{\lambda}{|y|})^{n+2}
\bar c(y^{\lambda})U(y^{\lambda}) + \bar c(y)(v_k^{\lambda}(y)-U^\lda(y))- (\frac{\lambda}{|y|})^{n+2}
\bar c(y^{\lambda})(v_k(y^{\lambda}) - U(y^\lda)) \\&
= (\Delta +V_\lda) f^{(3)}_{\lambda}(y)  +O(1) \begin{cases}M_k^{-1}( |S_k|^{-\frac{3}{2}}|y|^{-3}+|S_k|^{-4}|y|^{4-n}), & \quad \lda\le  |y|<\frac{\sigma}{2} |S_k|,\\[2mm]
M_k^{-1}|y|^{-\frac{n+2}{2}} , & \quad   |y|\ge \frac{\sigma}{2} |S_k|.
\end{cases}
\end{align*}
From the proof of Proposition \ref{e-e-lam}, and using $\sigma_k \le C|S_k|^{-\frac{n-2}{2}} $, which
is based on Lemma \ref{lem:sigma-k},  we know that
\begin{align*}
|\bar b_j(y)\partial_j v_k^{\lambda}(y)+\bar d_{ij}(y)\partial_{ij}v^{\lambda}_k(y)|&
\le C \sigma_k |x_k|^{\tau-2} l_k^{-2} |y|^{-n}\\&
\le C |S_k|^{-\frac{n-2}{2}}  M_k^{-1} |S_k|^{\frac{n-6}{2}} |y|^{-n}= C M_k^{-1} |S_k|^{-2} |y|^{-n}
\end{align*}
when $\lda \le |y|\le \frac{\sigma}{2} |S_k|$. By \eqref{bik-2} and \eqref{dijk-2}, we have
\[
|\bar b_j(y)\partial_j v_k^{\lambda}(y)+\bar d_{ij}(y)\partial_{ij}v^{\lambda}_k(y)| \le CM_k^{-1}|y|^{-\frac{n+2}{2}}
\]
when $|y|\ge \frac{\sigma}{2} |S_k|$.
Similarly,  we have
\begin{align*}
&(\frac{\lambda}{|y|})^{n+2}\left(\bar b_i(y^{\lambda})
\partial_i v_k(y^{\lambda})+
\bar d_{ij}(y^{\lambda})\partial_{ij}v_k(y^{\lambda})\right) \\&
\le C \begin{cases}M_k^{-1} |S_k|^{-2} |y|^{-2-n} , & \quad \lda\le  |y|<\frac{\sigma}{2} |S_k|,\\[2mm]
M_k^{-1}|y|^{-\frac{n+6}{2}} , & \quad   |y|\ge \frac{\sigma}{2} |S_k|.
\end{cases}
\end{align*}
Making use of \eqref{eq:724-i}, we complete the proof.

\end{proof}

Let $f_{\lda}^{(4)}(r)$  be the radial solution of
\be \label{eq:barrier-ode724}
-\Delta  f_{\lda}^{(4)} =-\frac{\ud^2}{\ud r^2 }f_{\lda}^{(4)} -\frac{n-1}{r} \frac{\ud}{\ud r } f_{\lda}^{(4)}
= Q M_k^{-1}|S_k|^{-1} r^{-3}, \quad r>\lda,
\ee
\[
f_{\lda}^{(4)}(\lda)=\frac{\ud}{\ud r }   f_{\lda}^{(4)}(\lda)=0,
\]
where $Q>C_2+2$ is a constant to be fixed. In fact,
\[
f_{\lda}^{(4)}(r)= Q M_k^{-1}|S_k|^{-1} \lda^{-1} \left (\frac{1}{n-3}\frac{\lda}{r} -\frac{1}{(n-2)(n-3)} (\frac{\lda}{r})^{n-2} -\frac{1}{n-2}\right)<0.
\]
\begin{lem} \label{lem:f4} By taking a large $Q$ independent of $k$, we have
\[
\left[\Delta +\bar b_i(y)\partial_i +\bar d_{ij}(y)\partial_{ij}-\bar c(y)+\xi_\lda(y)\right] f_{\lda}^{(4)}(|y|)\le - (C_2+1)  M_k^{-1}|S_k|^{-1}|y|^{-3}
\]
for $\lda \le |y|\le Q_1 l_k^{1/2}$.
\end{lem}

\begin{proof}
If $\lda \le |y|\le \sigma |S_k|$, by \eqref{ck-2} we have
\begin{align*}
&\left[ \Delta +\bar b_i(y)\partial_i +\bar d_{ij}(y)\partial_{ij}-\bar c(y) \right] f_{\lda}^{(4)}(|y|) \\
&= \Delta f_{\lda}^{(4)}(|y|)  -\bar c f_{\lda}^{(4)}(|y|) \\&
\le -  Q M_k^{-1}|S_k|^{-1}|y|^{-3} +C (|S_k|^{-4} |x_k^{\tau } ||y|^2) M_k^{-1}|S_k|^{-1}|y|^{-3}\\&
\le -(Q-1) M_k^{-1}|S_k|^{-1}|y|^{-3}.
\end{align*}

If $\sigma |S_k| \le |y|\le Q_1 l_k^{1/2}$, by \eqref{bik-2}, \eqref{dijk-2} and \eqref{ck-2}  we have
\begin{align*}
&\left[ \Delta +\bar b_i\partial_i +\bar d_{ij}\partial_{ij}-\bar c \right] f_{\lda}^{(4)}(|y|)  \\&
\le -  Q M_k^{-1}|S_k|^{-1}|y|^{-3} +C(M_k^{-1} |y|^{\tau}) M_k^{-1}|S_k|^{-1}|y|^{-3}\\&
\le -(Q-1) M_k^{-1}|S_k|^{-1}|y|^{-3}.
\end{align*}
 Since $\xi_\lda >0$ and $f_{\lda}^{(4)}(|y|)<0$, and $Q>2C_2+2$, the lemma follows immediately.

\end{proof}

\begin{prop} \label{cor:barrier-final}  Let $f_\lda(y)= f_{\lda}^{(3)}(y)+ f_{\lda}^{(4)}(|y|) $,
 where $\lda\in [1-\delta_1,1+\delta_1]$ and $\delta_1\in (0,1/2)$ is determined to make
 it possible to construct $f_{\lda}^{(3)}(y)$ from \eqref{f3sol}. Then
\[
\left[ \Delta +\bar b_i(y)\partial_i +\bar d_{ij}(y)\partial_{ij}-\bar c(y)+\xi_\lda(y) \right] (w_\lda+ f_{\lda})(y)\le 0
\]
for $y\in \tilde \Sigma^k:=\mathcal{O}_\lda \cup \{y: \lda\le |y|\le \frac{1}{2} \sigma |S_k|\}$,
  \[
f_\lda (y) =0 \quad \mbox{on }\pa B_\lda, \quad |f_\lda(y)| +|\nabla f_\lda(y)|=o(1)|y|^{2-n}
\text{  for } y\in  \tilde \Sigma^k_\lda,
\]
In particular, $|f_\lda(y)|=o(1) M_k^{-1} \quad \mbox{on } \pa B_{Q_1 l_k^{1/2}}$.
\end{prop}

\begin{proof}
The differential inequality follows from  \eqref{prop:error-est-2} and Lemma \ref{lem:f4}. The boundary
condition $f_\lda (y) =0 \quad \mbox{on }\pa B_\lda$ follows from the construction of $f_\lda(y)$.
To obtain the estimate for $|f_\lda(y)| +|\nabla f_\lda(y)|$, we apply \eqref{eq:724-e} to obtain
 $|f_\lda^{(3)}(y)|+|\nabla f_\lda^{(3)}(y)| \le C |x_k|^{\tau} (1+|y|)^{2-n}$. Furthermore,  the estimate for $|f_\lda^{(4)}(y)|$
 implies that $|f_\lda^{(4)}(y)|+|\nabla f_\lda^{(4)}(y)|=O(|S_k|^{-1})M_k^{-1}=O(|S_k|^{-1})l_k^{-\frac{n-2}{2}}$, but
 for $\lda\le |y|\le Q_1 l_k^{1/2}$, we certainly have $l_k^{-\frac{n-2}{2}}=O(|y|^{2-n})$. Since
 $|x_k|, |S_k|^{-1}\to 0$ as $k\to \infty$,  we thus conclude that
 $|f_\lda^{(4)}(y)|+|\nabla f_\lda^{(4)}(y)|=o(1) |y|^{2-n}$ for $\lda\le |y|\le Q_1 l_k^{1/2}$. Finally, on
 $\pa B_{Q_1 l_k^{1/2}}$, $|y|^{2-n}=Q_1^{2-n} l_k^{-\frac{n-2}{2}}$, we see that
 $|f_\lda(y)|=o(1) M_k^{-1}$.
\end{proof}

\section{The lower bound and removability of the singularity}
\label{sec:lb}

Suppose that $u$ is a solution of \eqref{eq:main-isolated} with $g$ satisfying \eqref{eq:inner-flat}.
If $n\ge 25$, we assume further that \eqref{eq:25-wk-bd-1} holds.
It follows from Theorem \ref{main-ub} that
\be \label{eq:upper-bound-7h}
u(x)\le C|x|^{-\frac{n-2}{2}} \quad \mbox{for } x\in B_1,
\ee
where $C>0$ is independent of $x$.

\begin{lem} \label{lem:sphere-harnack} Suppose that $u$ is a solution of \eqref{eq:main-isolated} with $g$ satisfying \eqref{eq:inner-flat}. If $n\ge 25$, we assume further that \eqref{eq:25-wk-bd-1} holds.   Then
\[
\max_{r/2\le |x|\le 2r } u(x)\le C_3 \min _{r/2\le |x|\le 2r } u(x)
\]
and
\[
|\nabla u(x)| +|x||\nabla^2 u(x)| \le C_3|x|^{-1} u(x)
\]
for every $0<|x|=r<1/4$, where $C_3$ is independent of $r$.

\end{lem}

\begin{proof}  For any $\bar  x\in B_{1/4}\setminus \{0\}$, let $r=|\bar x|$ and
\[
v_r(y)= r^{\frac{n-2}{2}} u(r y).
\]
By \eqref{eq:main-isolated}, we have
\[
-L_{\bar g}v_r=n(n-2)v_r(y)^{\frac{n+2}{n-2}}=0 \quad \mbox{ in } B_{1/r}\setminus \{0\},
\]
where $\bar g_{ij}(y)=g_{ij}(ry)$.  Since $g$ satisfies \eqref{eq:inner-flat}, we have
$|\nabla ^l(\bar g_{ij}(y)-\delta_{ij})| \le r^{\tau}$ for  $\frac14 \le |y|\le 4$, $l=0,\dots, n+2$.
By \eqref{eq:upper-bound-7h}, $v_r(y)\le C$ for $y\in B_{4}\setminus \bar B_{1/4}$.
Applying the standard local estimates and Harnack inequality to  $v_r$
in the annulus $B_{4}\setminus \bar B_{1/4}$ and scaling back to $u$, the lemma follows immediately.

\end{proof}

Recall that the Pohozaev identity for $u$ is
\begin{equation}\label{p-u1}
P(r,u)-P(s,u)=-\int_{B_r\setminus B_s}(\frac{n-2}2u+x\cdot \nabla u)(L_g-\Delta)u\,\ud x,
\end{equation}
 where $0<s\le r<1$,
$$P(r,u)=\int_{\partial B_r}\left(\frac{n-2}2u\frac{\partial u}{\partial r}-\frac 12 r|\nabla u|^2
+r|\frac{\partial u}{\partial r}|^2+\frac{(n-2)^2}2ru^{\frac{2n}{n-2}}\right)\,\ud S_r,$$
and  $\ud S_r$ is the standard area measure on $\pa B_r$.
By Lemma \ref{lem:sphere-harnack} and the flatness condition \eqref{eq:inner-flat} on $g$,  we have
$$|(\frac{n-2}2 u+x\cdot \nabla u)(L_g-\Delta )u|\le C|x|^{\tau-n}. $$
As in previous sections, we take $\tau =\frac{n-2}{2}$.
It follows that  $|P(r,u)-P(s,u)|\le Cr^\tau$ for any $0<s<r$ and hence  the limit
\be  \label{eq:pohozaev-lim}
\lim_{r\to 0}P(r,u)=:P(u)
\ee
exists.
\begin{lem} \label{lem:poho-crit} Assume as in Lemma \ref{lem:sphere-harnack}. Then we have
\[
P(u) \le 0.
\]
Moreover, $P(u)=0$ if and only if $\liminf_{x\to 0} |x|^{\frac{n-2}{2}}u(x)=0$.

\end{lem}

\begin{proof} Let $\{r_k\}$ be any sequence of positive numbers satisfying $\lim_{k\to \infty} r_k=0$, and let
\[
v_{k}(y)= r_k^{\frac{n-2}{2}} u(r_k y).
\]
By the proof of Lemma \ref{lem:sphere-harnack}, we see that, up to passing to a subsequence,
\[
v_k \to v \quad \mbox{in }C_{loc}^2 (\R^n\setminus \{0\}),
\]
where
\[
-\Delta v= n(n-2) v^{\frac{n+2}{n-2}} \quad \mbox{in }\R^n\setminus \{0\}, \quad v\ge 0.
\]
It follows that
\[
P(1, v)= \lim_{k\to \infty} P(1, v_k)= \lim_{k\to \infty} P(r_k, u)= P(u).
\]
By \cite{CGS}, $P(r,v)$ is a nonpositive constant independent of $r$ and $P(1,v)=0$ if and only if $v$ is smooth cross $\{0\}$. Clearly, if $P(u) =0$ then $\liminf_{x\to 0} |x|^{\frac{n-2}{2}}u(x) =0$; otherwise $v_k (y)\ge \frac{1}{C}|y|^{-\frac{n-2}{2}}$ for some $C>0$ independent of $k$ and thus $v$ is singular at $0$.

On the other hand, if $\liminf_{x\to 0} |x|^{\frac{n-2}{2}}u(x) =0$, by Lemma \ref{lem:sphere-harnack} we can find $r_k\to 0$ such that $v_k$ defined as above has trivial limit $v\equiv 0$. Hence $P(1,v)=0=P(u)$.

\end{proof}

Let $\bar u(r)= \dashint_{\partial B_r}u\,\ud S_r$ be the average of $u$ on $\partial B_r$.  Let $t=-\ln r$ and $ \bar u(r)=e^{\frac{n-2}{2} t} w(t) $. By a direct computation,
\[
\bar u_r= -e^{\frac{n}{2} t}\left(\frac{n-2}{2} w+ w_t\right),
\]
\[
\bar u_{rr}=e^{\frac{n+2}{2} t}\left(\frac{n(n-2)}{4} w+(n-1) w_t +w_{tt}\right).
\]
Therefore, we have
\[
\bar u_{rr}+\frac{n-1}{r} \bar u_r = e^{\frac{n+2}{2} t} \left(w_{tt}-(\frac{n-2}{2})^2 w\right).
\]
By Lemma \ref{lem:sphere-harnack} and  the flatness condition \eqref{eq:inner-flat} on $g$, we have
\begin{equation}\label{key-i}
-c_1w^{\frac{n+2}{n-2}}-c_3e^{-\tau t}w\le w_{tt}-(\frac{n-2}2)^2w\le -c_2w^{\frac{n+2}{n-2}}+c_3e^{-\tau t}w.
\end{equation}

\begin{prop} \label{prop:liminf-lim}Assume as in Lemma \ref{lem:sphere-harnack}. If
$\liminf_{x\to 0} |x|^{\frac{n-2}{2}}u(x) =0$, then
\be \label{eq:lim-0}
\lim_{x\to 0} |x|^{\frac{n-2}{2}}u(x) =0
\ee
and $0$ is a removable singularity.
\end{prop}

\begin{proof}  Our proof will largely follow the approach initiated by Chen-Lin \cite{ChenLin3}, but we need to prove
some necessary bounds in our context, as given later in Lemma \ref{lem:bessel} and Lemma \ref{lem:bessel-nonlinear}.
We argue by contradiction.
Since $\liminf_{x\to 0} |x|^{\frac{n-2}{2}}u(x) =0$, by Lemma \ref{lem:sphere-harnack}, if
the conclusion of the Proposition does not hold, we would have  $\limsup_{x\to 0}u(x)|x|^{\frac{n-2}2}>0$, so
\[
\limsup_{t\to \infty} w(t)>\liminf_{t\to \infty} w(t)=0.
\]
Making use of \eqref{key-i}, $w$ is convex ($w''>0$) when  $w$ small and $t$ is large. It follows that there exist
\[
\bar t_i<t_i<t_i^*
 \quad \mbox{with}\quad  \lim_{i\to \infty } \bar t_i= \infty,
 \] such that
 \[
 w(\bar t_i)=w(t_i^*)=\epsilon_0, \quad \lim_{i\to \infty} w(t_i)= 0
 \]
\[
t_i \mbox{ is the unique minimum point of } w \mbox{ in } (\bar t_i,t_i^*),
\] where $\epsilon_0>0$ is a small constant. Hence, $w$ is decreasing in $(\bar t_i,t_i)$ and increasing in $( t_i,t_i^*)$.

Using \eqref{key-i}, we will prove the following estimates
\begin{equation}\label{t-large}
\frac 2{n-2}\ln\frac{w(t)}{w(t_i)}-C\le t-t_i\le \Big(\frac{2}{n-2}
+Ce^{ -\tau t_i}\Big)\ln \frac{w(t)}{w(t_i)}+C,\quad  t_i\le t\le t_i^*,
\end{equation}
and
 \begin{equation}\label{bad-t}
\frac 2{n-2}\ln \frac{w(t)}{w(t_i)}-C\le t_i-t\le \Big(\frac{2}{n-2} \Big)\ln \frac{w(t)}{w(t_i)}+C',\quad \bar t_i\le t\le t_i
\end{equation}
where $C, C'>0$ is independent of $i$.

Marques obtained a cruder version of these estimates in \cite{f-mar},
replacing $\frac 2{n-2}$ on the left inequality in \eqref{t-large} by $\frac 2{n-2}- c e^{-\tau t_i}$
with $\tau=2$ and  some $c>0$, and similarly,
replacing $\frac 2{n-2}$ on the left inequality in \eqref{bad-t} by $\frac 2{n-2}- c e^{-\tau \bar t_i}$, and
$\frac 2{n-2}$ on the right inequality in \eqref{bad-t} by
$\frac 2{n-2}+ c e^{-\tau \bar t_i}$. His proof was based on that of Chen-Lin in \cite{ChenLin3},
which  estimates the $e^{-\tau t}$ factor in  \eqref{key-i} by its value at the left
end of the interval and estimates the resulting differential inequality as one with  constant
coefficients. His estimates
are adequate to handle his cases of $3\le n \le 5$, but are not adequate for our cases.
We will provide our proof for \eqref{t-large} and \eqref{bad-t} in the next section. The key is to
tackle the $e^{-\tau t}w$ term directly. The behavior exhibited
 is somewhat analogous to the behavior as in \cite{Lev}---see also
Theorem 8.1 (and Problems 29--31) in Chapter 3 of \cite{CL}, but our proof
does not rely on the method for treating linear systems as in \cite{Lev},
instead relies on some comparison principles.

We denote $r=|x|$,
\[
\bar r_i= e^{-\bar t_i}, \quad r_i=e^{-t_i},\quad r_i^*= e^{-t_i^*}.
\]
Thus $\bar r_i >r_i>r_i^*$.

First, we can compute $P(r, u)$  in terms of
$v(t,\theta) :=e^{-\frac{n-2}{2} t} u(e^{-t}\theta)$ as
\[
P(r, u)=\frac{|\mathbb{S}^{n-1}|}{2} \int_{ \mathbb{S}^{n-1}} \left[ v_t^2(t,\theta)-|\nabla_{\theta} v(t, \theta)|^2
-(n-2)^2\left(\frac{v(t,\theta)^2}{4}-v(t,\theta)^{\frac{2 n}{n-2}} \right) \right] d\theta,
\]
and using the Harnack and gradient estimates on $u(x)$ as given by Lemma  \ref{lem:sphere-harnack}
we see that, in terms of $v(t,\theta)$, we have
\[
|\nabla v(t, \theta)| =O(1) w(t),
\]
uniformly for $\theta\in \mathbb S^{n-1}$, so it follows that
\[
\int_{ \mathbb{S}^{n-1}} \left[ v_t^2(t_i,\theta)-|\nabla_{\theta} v(t_i, \theta)|^2\right]\, d\theta \to 0 \quad
\text{ as $i\to \infty$,}
\]
and
\begin{equation}
\label{eq:poho-0}
P(u)=\lim _{i \to \infty} P(r_{i}, u)=0.
\end{equation}

Next we  claim a more precise estimate for $|\nabla v(t_i, \theta)|$ at
$t=t_i$, which implies an equivalent estimate for $u(x)$ on $|x|=r_i$.
\begin{equation} \label{eq:asy-sym}
|\nabla v(t_i, \theta)|=o(1) w(t_i)\quad\text{uniformly for $\theta\in \mathbb S^{n-1}$.}
\end{equation}
Indeed, let $\zeta_i(y)=\frac{u(r_iy)}{u(r_i e_1)}$, where $e_1=(1,0,\dots, 0)$.  We have
$$
-L_{g_i}\zeta_i=n(n-2) \left(r_i^{\frac{n-2}2}u(r_i e_1)\right)^{\frac{4}{n-2}} \zeta_i^{\frac{n+2}{n-2}} \quad
 \mbox{in }B_{1/r_i}\setminus \{0\},
$$
where $ (g_i)_{kl}=g_{kl}(r_iy)$.  By Lemma \ref{lem:sphere-harnack},   $\zeta_i$ is locally uniformly bounded in $\R^n \setminus \{0\}$. By the choice of $r_i$, $r_i^{\frac{n-2}2}u(r_i e_1)\to 0$ as $i\to \infty$. Hence, $\zeta_i \to \zeta$ in $C^2_{loc}(\R^n \setminus \{0\}) $ for some $\zeta$ satisfying
\[
-\Delta \zeta=0 \quad \mbox{in }\R^n \setminus \{0\}, \quad \zeta\ge 0,
\]
$\zeta(e_1)=1$ and $\partial_r\left( \int_{ \mathbb{S}^{n-1}} \zeta(r\theta)r^{\frac{n-2}2}\, d\theta\right)=0$ at $r=1$,
which is based on the choice of $r_i$ in the definition for $\zeta_i(y)$.
 By the B\^ocher theorem and the two normalizing conditions above,
$\zeta(y)=a|y|^{2-n}+b$ with  $a=b=\frac 12$.   In terms of $v(t,\theta)$, this
is implying that $v(t_i+\tau, \theta)/v(t_i, e_1)$ converges to $\cosh(\tau)$ uniformly
on any compact interval of $\tau \times \mathbb S^{n-1}$, and its derivatives converge to
the respective ones of $\cosh(\tau)$. This in particular implies that
$\nabla_{\theta} v(t_i+\tau, \theta)/v(t_i, e_1)\to 0$ uniformly
at $\{\tau\}\times \mathbb S^{n-1}$ at any $\tau$, and $\nabla_{t} v(t_i, \theta)/v(t_i, e_1)\to 0$
uniformly for $\theta\in  \mathbb S^{n-1}$. Hence, \eqref{eq:asy-sym} follows.

Making use of \eqref{eq:asy-sym}, we have
$|v(t_i,\theta)-w(t_i)|=o(1)w(t_i)$ and $|\nabla_{t, \theta} v(t_i, \theta)|=o(1) w(t_i)$
uniformly for $\theta\in  \mathbb S^{n-1}$, so
\begin{align*}
P(r_{i}, u)=&|\mathbb{S}^{n-1}|\Big[
-\frac{1}{2}\left(\frac{n-2}{2}\right)^{2} w^{2}\left(t_{i}\right)(1+o(1)) \\ & +\frac{(n-2)^{2}}{2} w^{\frac{2 n}{n-2}}\left(t_{i} \right) (1+o(1)) \Big].
\end{align*}
Hence for sufficiently large $i$
\begin{equation}\label{temp-new-1}
w^{2}(t_{i}) \leq c_{n}|P(r_i,u )|.
\end{equation}

It follows from the Pohozaev identity (\ref{p-u1}) and \eqref{eq:poho-0} that
\begin{equation*}
\begin{split}
|P(r_i,u)|&\leq \int_{B_{r_{i}} \setminus  B_{r_{i}^{*}}}|\mathcal{A}(u)| \ud x +\int_{B_{r_{i}^{*}}}|\mathcal{A}(u)| \ud x \\&
=: I_1+I_2,
\end{split}
\end{equation*}
where
\[
\mathcal{A}(u)=\left(x^{k} \partial_{k} u+\frac{n-2}{2} u\right) (L_{g}u-\Delta u).
\]
By Lemma \ref{lem:sphere-harnack}, we have
\[
|\mathcal{A}(u)| \le C |x|^{\tau-n}.
\]
Hence,
\[
I_2 \le C(r_i^*)^{\tau} = C e^{-\tau t_i^*}.
\]
By the first inequality in (\ref{t-large}), we have
\[
w(t)\le C w(t_i) \exp \left( \Big(\frac{n-2}{2}\Big)(t-t_i)\right), \quad t_i\le t\le t_i^*,
\]
which implies
\[
u(x) \le C w(t_i) e^{-\tau t_i} |x|^{2-n}\quad \mbox{for } r_i^*\le |x|\le r_i.
\]
By Lemma \ref{lem:sphere-harnack}, we also have
\begin{align*}
|\mathcal{A}(u)(x)|& \le C  u(x) ^2|x|^{\tau-2}.
\end{align*}
Hence,
$$
I_1 \le C w(t_i)^2 e^{-(n-2) t_i}  \int_{r_i^*\le |x|\le r_i }  |x|^{2+\tau-2n }  \,\ud x
\le c w(t_i)^2e^{(2-n)t_i}(r_i^*)^{1-\frac n2}.
$$
By (\ref{t-large}) and (\ref{bad-t}), we see that
\[
t_i^*-t_i \le \Big(\frac{2}{n-2} +Ce^{ -\tau t_i}\Big)  \ln \frac{\va_0}{ w(t_i)} +C,  \quad
\frac{2}{n-2} \ln \frac{\va_0}{ w(t_i)}
\le  t_i-\bar t_i  + C.
\]
Hence,
\be \label{eq:t*}
t_i^*-t_i \le \Big(\frac{2}{n-2} +Ce^{ -\tau t_i}\Big) \Big(  t_i-\bar t_i  + C\Big)\frac{n-2}{2}+C
\le t_i-\bar t_i +C'',
\ee
for some constant $C''>0$ independent of $i$, where we have used
$e^{ -\tau t_i}  ( t_i-\bar t_i)\le e^{ -\tau t_i}  t_i$ is bounded above independent of $i$.
Using (\ref{eq:t*}) we can estimate $I_1$ more precisely:
\begin{align*}
I_1 &\le C w(t_i)^2 e^{-(n-2) t_i}  (r_i^*)^{1-\frac{n}2}\\&
= C w(t_i)^2 e^{-(n-2) t_i+(\frac{n-2}2)t_i^*}  \\&
\le C w(t_i)^2 e^{-\tau \bar t_i}.
\end{align*}
 Combining the estimates of $I_1$ and $I_2$, we have,
\be \label{temp-new-2}
|P(r_i,u)|\leq C w(t_i)^2  e^{-\tau \bar t_i}   +C  e^{-\tau t_i^*}.
\ee

Using \eqref{temp-new-1} and \eqref{temp-new-2}, we can combine terms to obtain $w(t_i)^2 \le  C e^{-\tau t_i^*}$, which is
\be \label{eq:vanishing}
\log \frac{1}{w(t_i)}\ge \frac{n-2}4t_i^*-C.
\ee
From the first inequality of \eqref{bad-t} and the first inequality of \eqref{t-large},  we have
\[
t_i-\bar t_i \ge \frac 2{n-2} \ln \frac{\va_0}{ w(t_i)}-C
\]
and
\[
t_i^*- t_i \ge \frac 2{n-2} \ln \frac{\va_0}{ w(t_i)}-C.
\]
Adding them up and using \eqref{eq:vanishing} and \eqref{eq:t*},  we have
\[
t_i^*- \bar t_i  \ge \frac{4}{n-2} \ln  \frac{1}{w(t_i)} -C\ge t_i^*-C,
\]
which implies $\bar t_i \le C$.
This contradicts $\bar t_i \to \infty$. Therefore, \eqref{eq:lim-0} holds.

Based on \eqref{eq:lim-0}  we clearly have $w'(t)<0$ for $t>T_1,$ where $T_i$ is sufficiently large.
Equation (\ref{key-i}) now implies
$$w_{tt}-(\frac{n-2}2-\delta)^2w\ge 0\quad \mbox{for } t\ge T_1, $$
where $\delta>0$ is some small constant. Thus for $t\ge T_1$,
$w_t^2-(\frac{n-2}2-\delta)^2w^2$ is non-increasing, non-negative due to $w(t)\to 0$ as
$t\to \infty$, and the integration of this quantity leads to
$$w(t)\le w(T_1)\exp(-(\frac{n-2}2-\delta)(t-T_1)),\quad t>T_1, $$
whose equivalent form is
$$u(x)\le C(\delta)|x|^{-\delta}. $$
Then standard elliptic estimate immediately implies that $u$ has a removable singularity at the origin.

Therefore, we complete the proof of Proposition \ref{prop:liminf-lim}.
\end{proof}

\begin{proof}[Proof of Theorem \ref{thm:pre-final}] It follows from Lemma \ref{lem:poho-crit} and Proposition \ref{prop:liminf-lim}.

\end{proof}

\begin{proof}[Proof of Theorem \ref{thm:pre-final-1}] By Lemma \ref{lem:sphere-harnack}, we have
\begin{align*}
-u^{-\frac{n+2}{n-2}} \Delta u&= -u^{-\frac{n+2}{n-2}} L_g u -u^{-\frac{n+2}{n-2}} (\Delta- L_g) u\\&
=n(n-2) +O(|x|^{\frac{n+2}{2}} \cdot |x|^{\tau-2-\frac{n-2}{2}})\\&
=n(n-2) +O(|x|^\tau) \quad \mbox{as }x\to 0.
\end{align*}
Since $0$ is not a removable singularity, \eqref{eq:up-low} holds.
Then the theorem follows immediately from Theorem 1 of Taliaferro-Zhang \cite{talia}---one step in
its proof, (2.25), relies on a statement from \cite{CGS},
for which one can also appeal to Theorem 3 of Han-Li-Teixeira \cite{HLT}.
\end{proof}

\begin{proof}[Proof of Theorem \ref{thm:main} and Theorem \ref{thm:main-25}] By the discussion
in Section \ref{sec:2},  the study of asymptotical behavior at the infinity of solution of the
Yamabe equation \eqref{main-eq} with asymptotically flat metric  $g$ satisfying \eqref{eq:asym-flt}
is equivalent  to the study of isolated singularity of solutions of  \eqref{eq:main-isolated}
with $g$ satisfying \eqref{eq:inner-flat}, after taking a Kelvin transform. Hence,
Theorem \ref{thm:main} and Theorem \ref{thm:main-25} follow from Theorem \ref{thm:pre-final} and Theorem \ref{thm:pre-final-1}.

\end{proof}

\section{Details of improved ODE estimates \eqref{t-large} and \eqref{bad-t}}
\label{sec:ode}

First, we recall a comparison principle.

\begin{lem}\label{lem:m-p} Suppose $\eta(t)$ and $b(t)$ are $C^2$ functions defined on $[t_1,t_2]$, and $b(t)\ge 0$. Suppose that
\[
\eta''(t)-b(t) \eta (t)\ge 0
\]

If
\[
\eta(t_1)\ge 0\quad \mbox{and}\quad \eta'(t_1)\ge 0,
\]
then $\eta(t)\ge \eta(t_1) $ and $\eta'(t)\ge 0$ on $[t_1,t_2]$.

If
\[
\eta(t_2)\ge 0\quad \mbox{and}\quad \eta'(t_2)\le 0,
\]
 then $\eta(t)\ge \eta(t_2) $ and $\eta'(t)\le 0$ on $[t_1,t_2]$.

\end{lem}

\begin{proof} Suppose that $\eta(t_1)\ge 0$ and $\eta'(t_1)\ge 0$, but $\eta(\bar t) <\eta(t_1)$ for some $\bar t\in (t_1,t_2]$. Then there must be some point $t_1\le \hat t<\bar t$ such that $\eta(\hat t)=\eta(t_1)$ and $\eta (t)<\eta(t_1)$ for  $t\in (\hat t, \bar t)$.
By Hopf lemma, we have $\eta'(\hat t)<0$. Let
\[
t^* =\inf\{t_1\le t<\hat t: \eta(s)>\eta(\hat t)\ge 0, \eta'(s)<0 \quad \mbox{for }t<s<\hat t \}.
\]
Then either $t^*=t_1$ or $\eta'(t^*)=0$. In either case Hopf Lemma would apply at $t^*$ and imply that $\eta'(t)<0$, which would be a contradiction. We conclude that $\eta(t)\ge \eta(t_1)$. Now it follows that  $\eta''(t)\ge b(t) \eta (t)\ge 0$, so $\eta'$ is non-decreasing on $[t_1,t_2]$, which implies $\eta'(t)\ge \eta'(t_1)\ge 0$ on $[t_1,t_2]$.

If $\eta(t_2)\ge 0 $ and $ \eta'(t_2)\le 0$, the proof is similar. Therefore, Lemma \eqref{lem:m-p} is proved.
\end{proof}




The following lemma will be used to prove the left side inequalities in \eqref{t-large} and \eqref{bad-t}.

\begin{lem} \label{lem:bessel} Suppose that  $a,b,t_1,t_2$ and $\tau$ are positive numbers,  and $1<t_1<t_2$.
Suppose that  $\eta$ is a positive $C^2$ function defined on $[t_1,t_2]$ and satisfies
\be \label{eq:bes-1}
\eta''(t) -(a^2+b e^{-\tau t}) \eta(t) \le 0 \quad \mbox{for }t_1\le t\le t_2.
\ee
\begin{itemize}
\item[(i)] If $\eta'(t_2)\ge 0$,  then
\begin{equation}\label{eq:bes-2}
\eta(t)\le \eta(t_2) \frac{\cosh (a(t-t_2)-f(t))}{\cosh (f(t_2))}, \quad  t_1\le t\le t_2,
\end{equation}
where $f(t)=\frac{b}{2a\tau}e^{-\tau t}$. Consequently,
\[
t_2-t\ge \frac{1}{a} \ln \frac{\eta(t)}{\eta(t_2)} -C,
\]
where $C>0$ depends only on $a, b$ and $\tau$.

\item[(ii)] If $\eta'(t_1)\le 0$,  then
\begin{equation}\label{eq:bes-8}
\eta(t)\le \eta(t_1) \frac{\cosh (a(t-t_1)+f(t))}{\cosh (f(t_1))}, \quad  t_1\le t\le t_2,
\end{equation}
and
\[
t-t_1\ge \frac{1}{a} \ln \frac{\eta(t)}{\eta(t_1)} -C,
\]
where $C>0$ depends only on $a, b$ and $\tau$.

\end{itemize}

\end{lem}

\begin{proof} We remark that the left hand side of \eqref{eq:bes-1}  can be transformed into a Bessel type equation, after  the change of variable $x=e^{-t}$,
so we can apply Lemma \ref{lem:m-p} to $\eta(t)$
with a solution of this Bessel type equation. But here we give an explicit comparison function.

Let
\[
\zeta(t):= \frac{\eta(t_2)}{ \cosh (f(t_2))}\cosh (a(t-t_2)-f(t)) .
\]
Differentiating $\zeta$ and using $f'(t)=-\tau f(t)$, we find that
\[
\zeta'(t)= \frac{\eta(t_2)}{ \cosh (f(t_2))}(a+\tau f(t))\sinh(a(t-t_2)-f(t)),
\]
\begin{align*}
\zeta''(t)=\frac{\eta(t_2)}{ \cosh (f(t_2))}\left( - \tau^2 f(t) \sinh (a(t-t_2)-f(t)) +(a+\tau f(t))^2\cosh (a(t-t_2)-f(t)) \right).
\end{align*}
Hence,
\begin{align*}
&\zeta''(t) -(a^2+b e^{-\tau t}) \zeta(t) \\&= \frac{\eta(t_2)}{ \cosh (f(t_2))} \Big( ((a+\tau f(t))^2 -(a^2+b e^{-\tau t})  )\cosh (a(t-t_2)-f(t))  \\& \quad  - \tau^2 f(t) \sinh (a(t-t_2)-f(t)) \Big)\\&
\ge \frac{\eta(t_2)}{ \cosh (f(t_2))}  (2a\tau \frac{b}{2a\tau}-b )e^{-\tau t} \cosh (a(t-t_2)-f(t))=0,
\end{align*}
where we have used $\sinh (a(t-t_2)-f(t)) \le 0$ if $t\le t_2$. Moreover,
\[
\zeta(t_2)=\eta(t_2) \quad \mbox{and}\quad \zeta'(t_2)\le 0.
\]
By Lemma \ref{lem:m-p}, we have $\eta(t)\le \zeta(t)$ for $t\in [t_1,t_2]$. Therefore, \eqref{eq:bes-2} is proved.

The proof of \eqref{eq:bes-8} is the same. Therefore, Lemma \ref{lem:bessel} is proved.
\end{proof}

\begin{proof}[Proof of lower bounds in \eqref{t-large} and \eqref{bad-t}]
For the lower bound in  \eqref{t-large}, we apply (ii) of Lemma \ref{lem:bessel}, identifying
$t_1=t_i$ and $t_2=t_i^*$. For the lower bound in \eqref{bad-t}, we apply (i) of Lemma \ref{lem:bessel}, identifying
$t_1=\bar t_i$ and $t_2=t_i$.
\end{proof}


\begin{lem}\label{lem:bessel-nonlinear} Suppose that  $b, c, n, t_1,t_2$ and $\tau$ are positive numbers, $t^*\le t_1<t_2$ and $n>2$.
 Suppose that  $w$ is a positive $C^2$ function defined on $[t_1,t_2]$ and satisfies
\be \label{eq:bes-5}
w''(t) -(\frac{n-2}{2})^2(1-b e^{-\tau t}) w(t)  +\frac{n(n-2)c}{4} w(t)^{\frac{n+2}{n-2}}\ge 0,
\va_0\ge w(t), \quad \mbox{for }t_1\le t\le t_2.
\ee
Then there exist positive constants $t^*$, $\va_0$, $C_1$ and $C_2$,  depending only on $b, c, n$ and $\tau$,
 such that:
\begin{itemize}
\item[(i)] If   $w'(t_2)\le 0$,
 then there holds
\be\label{eq:nonlinear-ODE-up}
t_2-t\le (\frac{2}{n-2} +C_1 e^{-\tau t_1}) \ln \frac{w(t)}{w(t_2)} +C _2 \quad \mbox{for }t_1\le t\le t_2.
\ee
\item[(ii)] If $w'(t_1)\ge 0$,
then there holds
\be\label{eq:nonlinear-ODE-up-2}
t_2-t\le (\frac{2}{n-2} +C_1 e^{-\tau t_1}) \ln \frac{w(t)}{w(t_1)} +C _2 \quad \mbox{for }t_1\le t\le t_2.
\ee
\end{itemize}
\end{lem}

\begin{proof}  The estimates in this lemma were proved by Marques \cite{f-mar} for $\tau=2$, which in turn is based  on the argument of Chen-Lin \cite{ChenLin3}. These upper bounds are also adequate for our purposes, so
we only sketch a different proof following the previous proof of Lemma \ref{lem:bessel}.
The upper bound in our
\eqref{bad-t} is stronger than  that in \eqref{eq:nonlinear-ODE-up}, where our set up also gives us
$w'(t)<0$ for $t_1<t<t_2$. Although the stronger version is not needed for
our proof of the main theorems of this paper, we will describe how to prove it at the end.

Let
\[
\zeta(t)=B\cosh^{\frac{2-n}{2}} (a (t-\bar t)),
\]
where $a=\sqrt{1-Ae^{-\tau t_1}}$  for some $A>0$ to be fixed,
$B^{\frac{4}{n-2}}c=1$, and $\bar t\ge t_2$ such that $\zeta(t_2)=w(t_2)$.  Then by a direct calculation similar
to the one in the proof for the previous Lemma, we have

\begin{align*}
&\zeta''(t) -((\frac{n-2}{2})^2-b e^{-\tau t}) \zeta(t)  +\frac{n(n-2)c}{4} \zeta^{\frac{n+2}{n-2}} \\&
= \frac{(n-2)^2B}{4} \left\{ -\frac{n}{n-2}\frac{a^2-1}{ \cosh^{2} (a(t-\bar t))}  +(a^2-1+\frac{4}{(n-2)^2} b e^{-\tau t})\right\}   \cosh^{\frac{2-n}{2}} (a(t-\bar t))
 \\&\le \frac{(n-2)^2 B}{4}\Big( \frac{n}{n-2}A e^{-\tau t_1} C\va_0^{\frac{n-2}{2}}- A e^{-\tau t_1}
 +\frac{4}{(n-2)^2}b e^{-\tau t} \Big)\cosh^{\frac{2-n}{2}} (a(t-\bar t)) \\&
 \le 0 ,
\end{align*}
if we take $\va_0$ to be  sufficiently small so that $ \frac{n}{n-2} C\va_0^{\frac{n-2}{2}}<0.1$, say; and take $A$
such that $0.9 A>b$.  Moreover,
\[
\zeta'(t_2) \ge 0,
\]
since $\bar t\ge t_2$.

Let $z=w-\zeta$. It follows  that
\[
z''(t) -(\frac{n-2}{2})^2(1-b e^{-\tau t}) z(t)  +\frac{n(n+2)}{4} \xi(t)^{\frac{4}{n-2}} z(t)\ge 0,
\]
\[
z(t_2)=0, \quad z'(t_2)\le 0,
\]
where
$$n(n+2) \xi(t)^{\frac{4}{n-2}}=\left\{\begin{array}{ll}
\displaystyle{n(n-2)\frac{w(t)^{\frac{n+2}{n-2}}-\zeta(t)^{\frac{n+2}{n-2}}}{w-\zeta}}, \quad \mbox{if }\quad w(t)\neq \zeta(t),\\[2mm]
n(n+2)w(t)^{\frac{4}{n-2}},\quad \mbox{if}\quad w(t)=\zeta(t).
\end{array}
\right.
$$
Note that $\zeta(t), w(t)\le \va_0$ for $t_1\le t\le t_2$, so by taking $\va_0$ small and $t^*$ large, we then have
\[
(\frac{n-2}{2})^2(1-b e^{-\tau t})  -\frac{n(n+2)}{4} \xi(t)^{\frac{4}{n-2}}\ge 0.
\]
By Lemma \ref{lem:m-p}, we have $w-\zeta\ge 0$ in $[t_1,t_2]$. Hence,  \eqref{eq:nonlinear-ODE-up} follows immediately.

The proof of \eqref{eq:nonlinear-ODE-up-2} is the same.

Our proof for the upper bound in \eqref{bad-t} is a modification of the proof in \cite{f-mar}. We use
the set up in \eqref{eq:nonlinear-ODE-up}, identifying $\bar t_i=t_1$
and $t_i=t_2$. Multiplying both sides of \eqref{eq:bes-5}
by $2w'(t)<0$, we have
\[
\begin{split}
0 &\ge \left[ w'(t)^2 -(\frac{n-2}{2})^2(1-b e^{-\tau t})w(t)^2 +c (\frac{n-2}{2})^2 w(t)^{\frac{2n}{n-2}}\right]'
+ (\frac{n-2}{2})^2 b\tau w(t)^2\\
&\ge \left[ w'(t)^2 -(\frac{n-2}{2})^2(1-b e^{-\tau t})w(t)^2 +c (\frac{n-2}{2})^2 w(t)^{\frac{2n}{n-2}}\right]'
\end{split}
\]
so $w'(t)^2 -(\frac{n-2}{2})^2(1-b e^{-\tau t})w(t)^2 +c (\frac{n-2}{2})^2 w(t)^{\frac{2n}{n-2}}$ is non-increasing
in $t$, and
\[
\begin{split}
& w'(t)^2 -(\frac{n-2}{2})^2(1-b e^{-\tau t})w(t)^2 +c (\frac{n-2}{2})^2 w(t)^{\frac{2n}{n-2}} \\
\ge &  -(\frac{n-2}{2})^2(1-b e^{-\tau t_2})w(t_2)^2 +c (\frac{n-2}{2})^2 w(t_2)^{\frac{2n}{n-2}},\\
\end{split}
\]
from which we obtain
\[
\begin{split}
w'(t)^2 \ge & (\frac{n-2}{2})^2(1-b e^{-\tau t})\left[ w(t)^2-w(t_2)^2\right]
- c (\frac{n-2}{2})^2\left[ w(t)^{\frac{2n}{n-2}}- w(t_2)^{\frac{2n}{n-2}}\right] \\
&+ b (\frac{n-2}{2})^2 ( e^{-\tau t_2} - e^{-\tau t})w(t_2)^2.\\
\end{split}
\]
When $w(t)/w(t_2)\le 2$, then for $t^*$ sufficiently large, $\va_0>0$ sufficiently small, there  exists $0<\delta<1$
such that for $t<s<t_2$,
we have $w''(s)-(\frac{n-2}{2})^2(1-\delta)^2w(s)\ge0$, so by Lemma \ref{lem:m-p},
$w(s)\ge w(t_2) \cosh[\frac{n-2}{2}(1-\delta)(s-t_2)]$, from which we get
$2\ge \cosh[\frac{n-2}{2}(1-\delta)(s-t_2)]$.  This gives an absolute upper bound for $t_2-t'$, where
 $t'$ is defined by $t<t'<t_2$ such that $w(t')/w(t_2)=2$.

 For $t_1<t<t'$, set $\eta(t)=w(t)/w(t_2)$, then $\eta(t)\ge 2$, and
 \[
 \eta'(t)^2\ge (\frac{n-2}{2})^2\left\{ (1-b e^{-\tau t})\left[\eta(t)^2-1\right]
 -cw(t_2)^{\frac{4}{n-2}} \left[ \eta(t)^{\frac{2n}{n-2}}-1\right]
 + b  ( e^{-\tau t_2} - e^{-\tau t})\right\}.
 \]
For $t^*$ sufficiently large, we have
$b  ( e^{-\tau t_2} - e^{-\tau t})\ge -0.1 (1-b e^{-\tau t})$ and $ 1-b e^{-\tau t} \ge 1/2$, so we get
\[
\eta'(t)^2\ge (\frac{n-2}{2})^2  (1-b e^{-\tau t}) \left\{\eta(t)^2-1.1-2c w(t_2)^{\frac{4}{n-2}} \left[ \eta(t)^{\frac{2n}{n-2}}-1\right]\right\}.
\]
Now if we introduce $s=\frac{n-2}{2} \int_t^{t_2} \sqrt{ 1-b e^{-\tau \hat t}} d\hat t$, then we get
\[
\big| \frac{d \eta}{ds}\big|^2\ge \eta^2-1.1-2c w(t_2)^{\frac{4}{n-2}} \left[ \eta^{\frac{2n}{n-2}}-1\right].
\]
Set $s'=\frac{n-2}{2} \int_{t'}^{t_2} \sqrt{ 1-b e^{-\tau \hat t}} d\hat t$. Then, on the one hand,
\[
s-s'\le \int_{2}^{w(t)/w(t_2)} \frac{d\eta}{\sqrt{ \eta^2-1.1-2c w(t_2)^{\frac{4}{n-2}} \left[ \eta^{\frac{2n}{n-2}}-1\right]}}
\le \log \frac{w(t)}{w(t_2)} +C
\]
for some constant $C$, where the integral on the right is estimated in a similar way
as in \cite{ChenLin3} and \cite{f-mar}. On the other hand,
$s-s'=\frac{n-2}{2} \int_t^{t'} \sqrt{ 1-b e^{-\tau \hat t}} d\hat t= \frac{n-2}{2} \left[ t'-t +O(1)\right]$.
The upper bound in \eqref{bad-t} now follows.
\end{proof}

\bigskip

\noindent Z.-C. Han

\noindent Department of Mathematics, Rutgers University\\
Hill Center-Busch Campus, 110 Frelinghuysen Road, Piscataway, NJ 08854\\[1mm]
Email: \textsf{zchan@math.rutgers.edu}

\medskip

\noindent J. Xiong

\noindent School of Mathematical Sciences, Beijing Normal University\\
Beijing 100875, China\\[1mm]
Email: \textsf{jx@bnu.edu.cn}

\medskip

\noindent L. Zhang

\noindent Department of Mathematics, University of Florida\\
        358 Little Hall P.O. Box 118105\\
        Gainesville FL 32611-8105, USA\\[1mm]
Email: \textsf{leizhang@ufl.edu}

\end{document}